\documentclass[12pt,journal]{IEEEtran}
\onecolumn
\pdfoutput=1
\usepackage{mathrsfs}
\usepackage{amsmath}
\usepackage{graphicx}
\usepackage[justification=centering]{caption}
\usepackage{amssymb}
\usepackage{enumerate}
\usepackage{longtable,tabularx,float}
\usepackage{cite}
\usepackage{mathcomp}
\usepackage{supertabular}
\usepackage{stmaryrd}
\usepackage{color}
\usepackage{url}
\usepackage{tikz}
\usepackage{hyperref}
\usepackage{cleveref}
\usepackage[latin1]{inputenc}
\usepackage{ntheorem}

\newtheorem{corollary}{Corollary}[section]

\newtheorem{definition}{Definition}[section]
\newtheorem{lemma}{Lemma}[section]
\newtheorem{theorem}{Theorem}[section]
\newtheorem{problem}{Problem}
\newtheorem*{proof}{Proof}

\newtheorem{remark}{Remark}[section]

\newcommand{\M}{\mathcal {M}}
\newcommand{\N}{\mathcal {N}}
\renewcommand{\P}{\mathcal {P}}

\newcommand{\bbZ}{{\mathbb Z}}

\newcommand{\bbR}{{\mathbb R}}

\newcommand{\va}{{\mathbf{a}}}
\newcommand{\vb}{{\mathbf{b}}}
\newcommand{\vh}{{\mathbf{h}}}
\newcommand{\vo}{{\mathbf{o}}}
\newcommand{\vu}{{\mathbf{u}}}

\newcommand{\vr}{{\mathbf{r}}}
\newcommand{\vi}{{\mathbf{i}}}
\newcommand{\vx}{{\mathbf{x}}}
\newcommand{\vy}{{\mathbf{y}}}

\newcommand{\parallelogram}{\tikz[baseline] \draw( 0em, .1ex) -- ++( .6em, 0ex) -- ++( .2em, 1.2ex) -- ++( -.6em, 0ex) -- cycle;}
\newcommand{\Rmnum}[1]{\uppercase\expandafter{\romannumeral #1}}

\numberwithin{equation}{section}

\usepackage{xcolor}
\definecolor{cream}{RGB}{203, 237, 204}

\begin{document}

\title{Reconstruction of hypermatrices from subhypermatrices}

\author{Wenjie~Zhong~and~Xiande~Zhang
        \thanks{W. Zhong ({\tt zhongwj@mail.ustc.edu.cn}) is with the School of Mathematical Sciences, University of Science and Technology of China, Hefei, 230026, Anhui, China.}

\thanks{X. Zhang ({\tt drzhangx@ustc.edu.cn}) is with the School of Mathematical Sciences,
University of Science and Technology of China, Hefei, 230026, and with Hefei National Laboratory, University of Science and Technology of China, Hefei, 230088, China.}

\thanks{The research is supported by the National Key Research and
Development Programs of China 2020YFA0713100 and 2023YFA1010200, the NSFC
under Grants No. 12171452 and No. 12231014, and the Innovation Program for Quantum Science and Technology
(2021ZD0302902).}
}

\maketitle
\begin{abstract}
For a given $n$, what is the smallest number $k$ such that every sequence of length $n$ is determined by the multiset of all its $k$-subsequences? This is called the $k$-deck problem for sequence reconstruction, and has been generalized to the two-dimensional case -- reconstruction of $n\times n$-matrices from submatrices. Previous works show that the smallest $k$  is at most $O(n^\frac{1}{2})$ for sequences  and at most $O(n^\frac{2}{3})$  for matrices. We study this $k$-deck problem for general dimension $d$ and prove that,  the smallest $k$ is at most $O(n^\frac{d}{d+1})$ for reconstructing a $d$ dimensional hypermatrix of order $n$ from the multiset of all its subhypermatrices of order $k$.
\end{abstract}

\section{Introduction}
For a binary sequence of length $n$, its \emph{$k$-deck }is the multiset of all its subsequences of length $k$. The $k$-deck problem is to determine the smallest positive integer $k$ for a given $n$, denoted by $\kappa(n)$, such that all sequences of length $n$ have distinct $k$-decks, or say, each sequence of length $n$ can be \emph{reconstructed} by its $k$-deck. This problem was first raised by Kalashnik in 1973 in an information-theoretic study \cite{kalashnik1973reconstruction} about deletion channels and has been widely studied by many researchers in combinatorics, see \cite{manvel1991reconstruction,scott1997reconstructing,
krasikov1997reconstruction,foster2000improvement,choffrut1997combinatorics,
dudik2003reconstruction}. The best known upper bound $\kappa(n)\leq(2\sqrt{\ln 2}+\varepsilon)\sqrt{n}$ was given by Foster and Krasikov \cite{foster2000improvement}; and the best known lower bound $\kappa(n)\geq \exp{(c\sqrt{\ln n})}$ was given by Dud\'{\i}k and Schulman \cite{dudik2003reconstruction}, where two different sequences of length $n$ with the same $k$-deck were constructed when $k\leq \exp{(c\sqrt{\ln n})}$ for some constant $c$. 

The $k$-deck problem has got new attention recently due to its applications in the DNA-based data storage \cite{yazdi2017portable,golm2022gapped}. Further, it is closely related to two reconstruction problems initiated by Levenshtein in 1966, one is the sequence reconstruction problem \cite{wz7,wz8,wz9}, which reconstructs a sequence from the set of all its subsequences instead of the multiset of all its subsequences; and  another is the trace reconstruction problem \cite{batu2004reconstructing,mcgregor2014trace,nazarov2017trace,holden2020subpolynomial,chase2021separating,chase2021new}, which reconstructs a sequence from a set of its traces (random subsequences) with high probability.

In 2003, Dud\'{\i}k and Schulman \cite{dudik2003reconstruction} proposed the question: what can be said about higher-dimensional versions of the $k$-deck problem? The lower bound of Dud\'{\i}k and Schulman can be applied to higher-dimensional arrays by using the same sequences along a fixed direction, see \cite[Section 1.3]{kos2009reconstruction}.  In 2009,  K\'{o}s et al.\cite{kos2009reconstruction} generalized this problem from sequences to matrices, where two kinds of decks were introduced. For a binary matrix $A$ of order ${n\times n}$, define its $k$-deck as the multiset of all submatrices obtained by deleting $n-k$ rows and $n-k$ columns of $A$, that is, the multiset of all submatrices of order $k\times k$. If we delete rows and columns symmetrically, we obtain the multiset of all principal submatrices of order $k\times k$, and call it the {\emph{principal} $k$-deck} of $A$. Let $\kappa_2(n)$ and $\kappa_2^\text{p}(n)$ denote the smallest $k$ such that a matrix of order ${n\times n}$ can be reconstructed by its $k$-deck and principal $k$-deck, respectively. K\'{o}s et al. \cite{kos2009reconstruction} showed that for both cases, if $n$ is sufficiently large, then $\kappa_2(n),\kappa_2^\text{p}(n)\leq 38n^\frac{2}{3}$.

The $k$-deck problem of matrices has a close relation to the well-known graph reconstruction problem; see  Ulam's conjecture \cite{o1970ulam}. But graph reconstruction is more complicated since graph automorphisms are involved. If we consider graphs with labelled vertices, then  reconstructing a symmetric $0$-$1$ matrix {from} its principal $k$-deck is equivalent to reconstructing an ordinary graph from all induced subgraphs of order $k$. Similarly, reconstructing a matrix with its $k$-deck is equivalent to reconstructing a bipartite graph from all induced bipartite subgraphs with each part of size $k$.

Besides the above work, there is no further result about the $k$-deck problem for the higher-dimensional version. In this paper, we consider this problem for a general {higher dimension}, say dimension $d$, which {has} natural connections with the reconstruction problem for $d$-uniform hypergraphs or $d$-partite $d$-uniform hypergraphs with labelled vertices. For a fixed integer $d\geq 3$, we consider elements in $\{0,1\}^{n\times \cdots \times n}$, where $n$ repeats $d$ times. Such an element is called an $n^{\times d}$-\emph{hypermatrix}. Similar to matrices, we define the $k$-deck and principal $k$-deck for hypermatrices. Here, the word ``principal'' means that the set of deleted indices are the same for each of the $d$ directions. Denote $\kappa_d(n)$ and $\kappa_d^\text{p}(n)$ the smallest $k$ such that an $n^{\times d}$-hypermatrix can be reconstructed by its $k$-deck and principal $k$-deck, respectively. Then our main result is stated as follows.

\vspace{0.5cm}
\begin{theorem}\label{d_reconstruction} For any fixed dimension $d\geq3$,
when $n$ is sufficiently large, we have $\kappa_d(n)\leq d^{\frac{3}{2}d}n^\frac{d}{d+1}$ and $\kappa_d^\text{p}(n)\leq d^{\frac{3}{2}d}n^\frac{d}{d+1}$. That is, all $n^{\times d}$-hypermatrices can be reconstructed by their (principal) $k$-decks when $k=\Omega(n^\frac{d}{d+1})$.
\end{theorem}
\vspace{0.5cm}

Note that the upper bound in \Cref{d_reconstruction} for hypermatrices is consistent with that {for} matrix reconstruction and sequence reconstruction in terms of the exponent $\frac{d}{d+1}$ of $n$.

The basic idea in the proof of  \Cref{d_reconstruction} follows from the one in the $k$-deck problems for sequences and matrices. In \cite{krasikov1997reconstruction}, Krasikov and Roditty connected the $k$-deck problem for sequences with the existence problem of a polynomial $p(x)$ with a peak value in $[n]:=\{1,2,\ldots,n\}$. Here a value $p(i)$ with $i\in [n]$ is called \emph{a peak value} of $p(x)$ if $p(i)>\sum_{x\in [n]\backslash\{i\}}|p(x)|$.  They showed that if such a polynomial exists, then all sequences of length $n$ have distinct $k$-decks when $k\geq\deg p+1$, by counting the coordinate-wise sums of the elements in the $k$-decks. So to get a good upper bound of $\kappa(n)$, one needs to find for a given $n$ a polynomial $p(x)$ of the least degree with a peak value at some $i\in[n]$. This problem is related to the  Littlewood-type problems \cite{borwein1997littlewood,borwein1999littlewood} and the Prouhet-Terry-Escott (PTE) problem \cite{borwein2002prouhet}, which have been widely studied \cite{green2009littlewood,alpers2007two,caley2013prouhet,caley2012prouhet,raghavendran2019prouhet}.  In \cite{kos2009reconstruction},  K\'{o}s et al. generalized this idea to matrices, where a two-variable polynomial with a peak value and of the least possible degree was constructed to upper bound $\kappa_2(n)$ and $\kappa_2^\text{p}(n)$.

In this paper, we first connect the $k$-deck problem for hypermatrices to the existence problem of a $d$-variable polynomial $p(\vx)$ with a peak value {in any given} set $H\subset [n]^d$, i.e., $p(\mathbf{h})>\sum_{\mathbf{x}\in H\setminus\mathbf{h}}|p(\mathbf{x})|$ for some $\mathbf{h}\in H$. This is a natural $d$-dimensional analogue of sequences and matrices. Then the most technical part is to explicitly construct a $d$-variable  polynomial $p(\vx)$ with a peak value and of the smallest possible degree, which we think is interesting itself.
It turns out that nontrivial difficulties arise when we attempt to bound the total degree $\deg p=O(n^\frac{d}{d+1})$ of a $d$-variable polynomial  with a peak value. By using tools in lattice theory, we succeed in obtaining the following result.

\vspace{0.5cm}
\begin{theorem}\label{d_polynomial} Let $d\geq3$ be a
 fixed integer.
When $n$ is sufficiently large, for any nonempty set $H\subset [n]^d$, there exists a $d$-variable polynomial $p(x_1,\ldots,x_d)$ with $\deg p\leq d^{\frac{3}{2}d}n^\frac{d}{d+1}-d$ and a point $\vh=(h_1,\ldots,h_d)\in H$ such that
\begin{equation}\label{equation d}
p(\mathbf{h})>\sum_{\mathbf{x}\in H\setminus\mathbf{h}}|p(\mathbf{x})|.
\vspace{-0.2cm}
\end{equation}
\end{theorem}
\vspace{0.5cm}

The paper is organized as follows. In Section 2, we introduce the polynomial method to solve the $k$-deck problem for hypermatrices. Section 3 sketches how to construct the required polynomial; Section 4 gives the construction in detail; and Section 5 applies the polynomial to prove \Cref{d_polynomial}.

In Section 6, we revisit the matrix reconstruction, that is when $d=2$. By using Pigeonhole principle, we give a simpler construction of polynomials with a peak value and of a lower degree. This
improves the upper bounds $\kappa_2(n),\kappa_2^\text{p}(n)\leq 38n^\frac{2}{3}$ of K\'{o}s et al. in terms of the constant as below. 


\begin{theorem}\label{2_polynomial}
When $n$ is sufficiently large, for an arbitrary nonempty set $H\subset [n]^2$, there exists a polynomial $p(x,y)$ with $\deg p\leq 6.308n^\frac{2}{3}-2$ and a point $\mathbf{h}=(h_1,h_2)\in H$ such that
\begin{equation}\label{equation 2}
p(h_1,h_2)>\sum_{(x,y)\in H\setminus\mathbf{h}}|p(x,y)|.
\end{equation}Consequently, we have $\kappa_2(n),\kappa_2^\text{p}(n)\leq 6.308n^\frac{2}{3}$.
\end{theorem}

\vspace{0.5cm}
\section{The polynomial method}\label{sec-pol}
%

In this paper, we only consider binary (hyper)matrices. For general $q$-ary matrices, the problem can be reduced to the binary case, see \cite{manvel1991reconstruction}.

For a fixed integer $d\geq 3$, we write $\{0,1\}^{n^{\times d}}$ instead of $\{0,1\}^{n\times \cdots \times n}$ for convenience. An element in $\{0,1\}^{n^{\times d}}$ is called an $n^{\times d}$-hypermatrix. For an $n^{\times d}$-hypermatrix $A$, denote by $\M_k(A)$  the $k$-deck of $A$, that is the multiset of the $\binom{n}{k}^d$ many $k^{\times d}$-subhypermatrices of $A$. Similarly, denote by $\M_k^\text{p}(A)$ the principal $k$-deck of $A$, that is, the multiset of the  $\binom{n}{k}$ principal $k^{\times d}$-subhypermatrices of $A$.

Regarding $\M_k$  as a map on $\{0,1\}^{n^{\times d}}$, we see that every hypermatrix $A\in \{0,1\}^{n^{\times d}}$ is uniquely determined by its $k$-deck $\M_k(A)$ if and only if the map $\M_k$ is injective.
For $k<l$, $M_l(A)$ determines $\M_k(A)$ since $\biguplus_{B\in M_l(A)}\M_k(B)=\binom{n-k}{l-k}^d\cdot \M_k(A)$ (here $\biguplus$ denotes the multiset union). Then an injective $\M_k$ implies an injective $M_l$. Hence it is interesting to find the smallest $k$ such that $\M_k$ is injective, denoted by $\kappa_d(n)$ for given $d$ and $n$. The above argument works for $\M_k^{\text{p}}$ in the principal $k$-deck problem, where we denote the smallest $k$ by $\kappa_d^\text{p}(n)$.

Generalizing the ideas of Krasikov and Roditty \cite{krasikov1997reconstruction}, we define the sums of subhypermatrices  in $\M_k(A)$ and in $\M_k^{\text{p}}(A)$ as
\[S_k(A)=\sum_{B\in \M_k(A)}B~\text{and}~ {S_k^{\text{p}}(A)}=\sum_{B\in \M_k^{\text{p}}(A)}B.\]
It is clear that $S_k(A)$ and $S_k^{\text{p}}(A)$ are $k^{\times d}$-hypermatrices.
%

Similar to $\M_k$, we regard $S_k$ as a map on $\{0,1\}^{n^{\times d}}$. By definition, $\M_k$ determines $S_k$ and then an injective $S_k$ implies an injective $\M_k$. That is, if all $n^{\times d}$-hypermatrices are uniquely determined by their $S_k$, then they are uniquely determined by their $k$-decks $\M_k$. The fact simplifies the $k$-deck problem via $S_k$ reconstruction. Similar arguments work for $S_k^{\text{p}}$ and $\M_k^{\text{p}}$.
%

In this paper, our main work is to prove the following result for $S_k$ ($S_k^{\text{p}}$) reconstruction, which immediately implies \Cref{d_reconstruction}.
\vspace{0.3cm}
\begin{theorem}\label{S_k reconstruction} For any fixed dimension $d\geq3$,
when $n$ is sufficiently large, all $n^{\times d}$-hypermatrices can be reconstructed by both of their $S_k$ and $S_k^{\text{p}}$ when $k\geq d^{\frac{3}{2}d}n^\frac{d}{d+1}$.
\end{theorem}
\vspace{0.5cm}

The proof of Theorem~\ref{S_k reconstruction} will be given in the end of this section. Similar to \cite[Theorem 3.1]{kos2009reconstruction}, the following result shows that the bounds of $k$ in \Cref{S_k reconstruction} are close to be optimal by simple counting.
\vspace{0.3cm}
\begin{lemma}\label{S_k lower bound}
Given the dimension $d\geq3$, if $k\leq\frac{n^{d/(d+1)}}{\sqrt[d+1]{d\log_{2}(n+1)}}$, then there exist hypermatrices that can not be reconstructed by their $S_k$ or $S_k^{\text{p}}$. That is, there exist $A,A',B,B'\in\{0,1\}^{n^{\times d}}$ such that $A\neq A'$ but $S_k(A)=S_k(A')$, and $B\neq B'$ but $S_k^{\text{p}}(B)=S_k^{\text{p}}(B')$.
\end{lemma}
\begin{proof}
 For an arbitrary hypermatrix $A\in\{0,1\}^{n^{\times d}}$, $S_k(A)$ is the sum of $\binom{n}{k}^d$ subhypermatrices, so each entry in $S_k(A)$ is a nonnegative integer with value at most $\binom{n}{k}^d$. Hence
\[|\{S_k(A):A\in\{0,1\}^{n^{\times d}}\}|\leq\left(\binom{n}{k}^d+1\right)^{k^d}\leq(n^{dk}+1)^{k^d}<(n+1)^{dk^{d+1}}\leq2^{n^d}=|\{0,1\}^{n^{\times d}}|.\]
There are fewer possible choices of $S_k(A)$ than binary $n^{\times d}$-hypermatrices, so the map $S_k$ cannot be injective.

 Similarly for the symmetric case, each entry of $S_k^{\text{p}}(A)$ is at most $\binom{n}{k}$ and therefore
\[|\{S_k^{\text{p}}(A):A\in\{0,1\}^{n^{\times d}}\}|\leq\left(\binom{n}{k}+1\right)^{k^d}\leq(n^k+1)^{k^d}\leq(n+1)^{k^{d+1}}<2^{n^d}=|\{0,1\}^{n^{\times d}}|.\]
\end{proof}

\vspace{0.5cm}

\Cref{S_k reconstruction} and \Cref{S_k lower bound} show that the smallest value of $k$, for which
{$S_k$ or $S_k^{\text{p}}$ determines all $n^{\times d}$-hypermatrices,} is between $\Omega(\frac{n^{d/(d+1)}}{\sqrt[d+1]{\log_{2}n}})$ and $O(n^{d/(d+1)})$. This means that the exponent $d/(d+1)$ is sharp in the $S_k$ ($S_k^{\text{p}}$) reconstruction problem. However, the bound for $S_k$ ($S_k^{\text{p}}$) in Lemma~\ref{S_k lower bound} does not work for $\M_k$ ($\M_k^{\text{p}}$) in general.
%

Next, we show how to turn the $S_k$ ($S_k^{\text{p}}$) reconstruction problem, that is the proof of \Cref{S_k reconstruction}, to a polynomial construction problem. We  discuss the two cases  $S_k$ and $S_k^{\text{p}}$ separately.

For a $d$-variable polynomial $p(x_1,\ldots,x_d)$, denote $\deg p$ the total degree of $p$, and $\deg_l p$ the local degree of $p$, that is the maximum degree among all variables.
For example, let $p=x_1^2x_2^2+x_2^3$, then $\deg p=4$ and $\deg_l p=3$. In this paper, we only consider polynomials with real coefficients.

\subsection{The nonsymmetric case}
The nonsymmetric case is a natural $d$-dimensional analogue of \cite[Lemma 2.1]{kos2009reconstruction}.

For $1\leq u\leq k$, define a polynomial $\beta_u(x):=\binom{x-1}{u-1}\binom{n-x}{k-u}$, then $\deg \beta_u(x)=k-1$. It is known that $\beta_u(x)$, $1\leq u\leq k$ form a basis of the linear space of all polynomials with degree less than $k$ \cite[Lemma 2.3]{kos2009reconstruction}. Denote $\vx:=(x_1,\ldots,x_d)$. For any $\vu=(u_1,\ldots,u_d)\in[k]^d$, define  $\beta_{\vu}(\vx):=\beta_{u_1}(x_1)\times\cdots\times\beta_{u_d}(x_d)$. Then $\beta_{\vu}(\vx), \vu\in[k]^d$ form a basis of the linear space of all $d$-variable polynomials with local degree less than $k$.

\vspace{0.3cm}
\begin{lemma}\label{nonsym}
Let $A,B\in \{0,1\}^{n^{\times d}}$ be two arbitrary hypermatrices and let $D=A-B=(d_{i_1\cdots i_d})_{1\leq i_1,\ldots,i_d\leq n}\triangleq(d_{\vi})_{\vi\in [n]^d}$ be their difference. Then $S_k(A)=S_k(B)$ if and only if
\begin{equation}\label{eq-p}
  \sum_{\vi\in[n]^d}p(\vi)\cdot d_{\vi}=0
\end{equation}
for any polynomial $p(\vx)$ with local degree $\deg_l p\leq k-1$.

\end{lemma}


\begin{proof} By definition, the $k^{\times d}$-hypermatrix $S_k(D)=S_k(A)-S_k(B)$. We first claim that for each $\vu=(u_1,\ldots,u_d)\in[k]^d$, the $\vu$-entry in $S_k(D)$ can be expressed as
\begin{equation}\label{eq-sk}
  \big(S_k(D)\big)_{\vu}=\sum_{\vi\in[n]^d}\beta_{\vu}(\vi)d_{\vi}.
\end{equation}
To prove this claim, it suffices to show that each  $d_{i_1\cdots i_d}$ occurs $\beta_{u_1}(i_1)\cdots\beta_{u_d}(i_d)$ times in $\M_k(D)$ as the $(u_1,\ldots,u_d)$-entry. Indeed, when we delete $n-k$ rows of the $j$-th dimension of $D$, there are $\beta_{u_j}(i_j)=\binom{i_j-1}{u_j-1}\binom{n-i_j}{k-u_j}$ choices such that the $i_j$-th row of $D$ becomes the $u_j$-th row of the subhypermatrix. By the independence of deletions in each dimension, each entry $d_{i_1\cdots i_d}$ occurs $\beta_{u_1}(i_1)\cdots\beta_{u_d}(i_d)$ times in $\M_k(D)$ as the $(u_1,\ldots,u_d)$-entry. This proves the claim.

Suppose that $S_k(A)=S_k(B)$, that is $ \big(S_k(D)\big)_{\vu}=0$ for any $\vu\in[k]^d$. Let $p(\vx)$ be an arbitrary polynomial with $\deg_l p\leq k-1$. Then there exist real numbers $\lambda_{\vu}, \vu\in[k]^d$ such that
\[p(\vx)=\sum_{\vu\in[k]^d}\lambda_{\vu}\beta_{\vu}(\vx).\]
Hence,
\begin{align*}
&\sum_{\vi\in[n]^d}p(\vi)\cdot d_{\vi}=\sum_{\vi\in[n]^d}\left(\sum_{\vu\in[k]^d}\lambda_{\vu}\beta_{\vu}(\vi)\right)d_{\vi}\\
=&\sum_{\vu\in[k]^d}\lambda_{\vu}\left(\sum_{\vi\in[n]^d}\beta_{\vu}(\vi)d_{\vi}\right)
\overset{(\ref{eq-sk})}{=}\sum_{\vu\in[k]^d}\lambda_{\vu}\cdot \left(S_k(D)\right)_{\vu}=0.
\end{align*}

Now we prove the other way.  Suppose Eq. (\ref{eq-p}) is true for any polynomial $p$ with local degree less than $k$. Then applying Eq. (\ref{eq-p}) with the polynomial $\beta_{\vu}(\vx)=\beta_{u_1}(x_1)\times \cdots \times \beta_{u_d}(x_d)$, which has local degree $k-1$, we have
\[\left(S_k(D)\right)_{\vu}\overset{(\ref{eq-sk})}{=}\sum_{\vi\in[n]^d}\beta_{\vu}(\vi)d_{\vi}\overset{(\ref{eq-p})}{=}0\]
for any $\vu\in [k]^d$. So $S_k(D)=0$, i.e., $S_k(A)=S_k(B)$.

\end{proof}

\subsection{The symmetric case}
The symmetric case is a $d$-dimensional analogue of \cite[Lemma 2.4]{kos2009reconstruction}, but more complicated.

We use ``$\twoheadrightarrow$" to mean a map is surjective. For $\tau:[d]\twoheadrightarrow[r]~(1\leq r\leq d)$, let $I_\tau\subset [n]^d$ consist of $n$-tuples which have the same relative orders as all of the digits in $(\tau(1),\ldots,\tau(d))$.
That is, $(i_1,\ldots,i_d)\in I_\tau$ if  for any $1\leq j_1<j_2\leq d$, we have $i_{j_1}>i_{j_2}$ if and only if $\tau(j_1)>\tau(j_2)$, and so are the cases $i_{j_1}<i_{j_2}$ and $i_{j_1}=i_{j_2}$. For example, when $n=4,d=3,r=2$ and $\tau([3])=(\tau(1),\tau(2),\tau(3))=(1,2,1)$, we have $I_\tau=\{(1,2,1),(1,3,1),(1,4,1),(2,3,2),(2,4,2),(3,4,3)\}\subset [4]^3$. It is easy to see that for two distinct $\tau_1:[d]\twoheadrightarrow[r_1]$ and $\tau_2:[d]\twoheadrightarrow[r_2]$ ($r_1=r_2$ is allowable), $I_{\tau_1}\cap I_{\tau_2}=\emptyset$ and
\begin{equation}\label{eq-nd}
  [n]^d=\cup_{r=1}^d\cup_{\tau:[d]\twoheadrightarrow[r]}I_\tau,
\end{equation}
i.e., $[n]^d$ is a disjoint union of all $I_\tau$.
%
\vspace{0.3cm}
\begin{lemma}\label{sym}
Let $A,B\in \{0,1\}^{n^{\times d}}$ be two arbitrary hypermatrices and let $D=A-B=(d_{\vi})_{\vi\in[n]^d}$ be their difference. Then $S_k^{\text{p}}(A)=S_k^{\text{p}}(B)$ if and only if for any integer $r$ such that $1\leq r\leq d$, we have
\begin{equation}\label{eq-it}
  \sum_{\vi\in I_\tau}p(\vi)\cdot d_{\vi}=0
\end{equation}
 for any surjective $\tau:[d]\twoheadrightarrow[r]$ and for  any polynomial $p$ with total degree $\deg p\leq k-r$.

\end{lemma}
\vspace{0.3cm}

The proof of \Cref{sym} is similar to that of \Cref{nonsym}, hence we move it to Appendix~\ref{app-1}.  Note that the summation in Eq. (\ref{eq-it}) is restricted on $I_\tau$. By Eq. (\ref{eq-nd}), it is immediate to have the following simpler version  of \Cref{sym}, which will be applied to prove \Cref{S_k reconstruction}.

\vspace{0.3cm}
\begin{corollary}\label{cor-sy}
Let $A,B\in \{0,1\}^{n^{\times d}}$ and $D=A-B=(d_{\vi})_{\vi\in[n]^d}$. If $S_k^{\text{p}}(A)=S_k^{\text{p}}(B)$, then
\begin{equation*}
  \sum_{\vi\in [n]^d}p(\vi)\cdot d_{\vi}=0
\end{equation*}
for any polynomial $p(\vx)$ with total degree $\deg p\leq k-d$.
\end{corollary}

\subsection{Proof of \Cref{S_k reconstruction}}
Combining  \Cref{d_polynomial} with \Cref{nonsym} and \Cref{cor-sy}, we can prove \Cref{S_k reconstruction} as follows, and thus prove \Cref{d_reconstruction}.

\vspace{0.3cm}
\begin{proof}[Proof of \Cref{S_k reconstruction}]

Let $k\geq d^{\frac{3}{2}d}n^\frac{d}{d+1}$. Given any two distinct $A,B\in \{0,1\}^{n^{\times d}}$, let $D=A-B=(d_{\vi})_{\vi\in[n]^d}$. Consider the set $H=\{\vi\in [n]^d:d_{\vi}\neq 0\}$. By \Cref{d_polynomial}, there exists a point $\mathbf{h}=(h_1,\ldots,h_d)\in H$ and a polynomial $p(\vx)$ with total degree $\deg p\leq k-d$ such that \begin{equation*}
p(\mathbf{h})>\sum_{\mathbf{x}\in H\setminus\mathbf{h}}|p(\mathbf{x})|.
\end{equation*}  Then
\begin{align*}
&\left|\sum_{\vi\in[n]^d}p(\vi)d_{\vi}\right|=\left|\sum_{\vi\in H}p(\vi)d_{\vi}\right|
\geq p(\mathbf{h})-\sum_{\vi\in H\setminus\mathbf{h}}|p(\vi)|>0.
\end{align*}
Hence, $\sum_{\vi\in[n]^d}p(\vi)d_{\vi}\neq 0$. By \Cref{nonsym} and \Cref{cor-sy}, this implies $S_k(A)\neq S_k(B)$ and $S_k^{\text{p}}(A)\neq S_k^{\text{p}}(B)$, respectively. That is, all $n^{\times d}$-hypermatrices can be reconstructed by their $S_k$ ($S_k^{\text{p}}$).
\end{proof}

The remaining of this paper is devoted to prove  \Cref{d_polynomial}, that is, to construct a $d$-variable polynomial with a peak value whose total degree is $O(n^\frac{d}{d+1})$.
\vspace{0.5cm}
\section{Construction sketch of our polynomial}

%

In this section, we give the idea of {the construction} of our required polynomial. Consider the $d$-dimensional Euclidean space $E^d$. A $d$-sphere denotes the set of points in $E^d$ that are situated at a constant distance from a fixed point. A $(d-1)$-hyperplane  is a translate of a $(d-1)$-dimensional subspace of $E^d$. 
{An $r$-lattice is a lattice of rank $r$.}

In the construction, a key ingredient is 
{the so-called direction function}.  A function $g: E^d \rightarrow \bbR$ is called a \emph{direction function} if $g(\vx)$ is a linear function of the form
\[g(\vx)=\va\cdot(\vx-\vh)=a_1(x_1-h_1)+a_2(x_2-h_2)+\cdots+a_d(x_d-h_d)\]
for some nonzero $\va=(a_1,a_2,\ldots,a_d)$ and $\vh=(h_1,h_2,\ldots,h_d)\in E^d$. We observe that $\deg g(\vx)=1$ and the value of $g(\vx)$ increases along the direction of $\va$, i.e., each $(d-1)$-hyperplane orthogonal to $\va$ is a level set of $g(\vx)$, so we call $\va$ the \emph{direction vector} of $g$. If we compose a direction function into a univariate function $f$, then the value of $f\circ g(\vx)$ changes along the direction of $\va$, where $g(\vx)$ plays a role of a direction.

In \cite{kos2009reconstruction}, K\'{o}s et al. set a bivariate polynomial as $p(x,y)=p_0(x,y)p_1(x,y)$ with $p_i(x,y)=f_i\circ g_i(x,y)$, $i=0,1$, where $g_i$ is a direction function and $f_i$ is a univariate polynomial. In their construction, $g_i(\vx)=\vu_i\cdot (\vx-\vh)$, $i=0,1$ for some carefully chosen $\vh$ and linearly independent vectors $\vu_0$ and $\vu_1$ in $E^2$, so that the map $(x,y)\mapsto(g_0(x,y),g_1(x,y))$ is injective. This allows them to estimate the sum as
\begin{equation}\label{eq-p2}
  \sum_{\vx\in H}|p(\vx)|=\sum_{\vx\in H}|f_0\circ g_0(\vx)||f_1\circ g_1(\vx)|\leq \sum_{z\in g_0(H)}|f_0(z)|\cdot\sum_{z\in g_1(H)}|f_1(z)|,
\end{equation}
for any set $H\subset [n]^2$, which degenerates the polynomial into the univariate case so that one can construct it by applying the polynomials for sequence reconstruction.

For the two univariate polynomials $f_i$,  K\'{o}s et al.  \cite{kos2009reconstruction} chose them as in the following lemmas which behave like a pulse, taking a large value somewhere and decaying quickly elsewhere. This property is useful for showing that the polynomial $p$ has a peak value and a low degree.\\
\vspace{0.5cm}
\begin{lemma}\label{f1}(\cite[Lemma 3.4]{kos2009reconstruction} )
For arbitrary reals $b_0,m_0>0$, there exists a polynomial $f(x)$ with real coefficients such that
\begin{itemize}
  \item[(a)]$f(0)=m_0$,
  \item[(b)]$|f(x)|\leq\min(m_0,\frac{1}{x^2})$ for all $x\in (0,b_0]$, and
  \item[(c)]$\deg f<\sqrt{\pi}\sqrt{b_0}\sqrt[4]{m_0}+2$.
\end{itemize}
\end{lemma}

\vspace{0.5cm}
\begin{lemma}\label{f2}(\cite[Lemma 3.5]{kos2009reconstruction} )
For arbitrary reals $a,b,m\geq 1$, there exists a polynomial $f(x)$ with real coefficients such that
\begin{itemize}
  \item[(a)]$f(0)=m$,
  \item[(b)]$|f(x)|<\min(4m,\frac{1}{x^2})$ for all $x\in [-a,b], x\neq 0$, and
  \item[(c)]$\deg f<7\sqrt{abm}+2$.
\end{itemize}
\end{lemma}

\vspace{0.5cm}

For the $d$-dimensional case, we also suppose the polynomial mapping $E^d$ to $\bbR$ to be of the form $p(\vx)=p_0(\vx)\cdots p_{d-1}(\vx)$, where $p_i(\vx)=f_i \circ g_i(\vx)$ and $g_i$ is a direction function. The univariate polynomial $f_i$ could be chosen from \Cref{f1} or \Cref{f2} as before. However, the direction functions $g_i$ should be very carefully chosen such that the values of $g_i$ match with the domain of $f_i$, so as to bound the degree of $f_i$ by \Cref{f1} or \Cref{f2}. To accomplish this, we need to introduce the concept of ``primitive point" and related lattices in $\bbZ^d$, to apply the tools of lattice theory.

\subsection{Notations}

In the Euclidean space $E^d$, we use a bold small letter to mean a point or the corresponding  vector, whose meaning could be easily recognized from the context. For example, let $\|\va\|$ denote the norm (length) of the vector $\va$, and let $\angle(\va,\vb)$ denote the angle between two vectors $\va$ and $\vb$. On the other hand, the notation $\va\vb$ denotes the line segment with two endpoints $\va$ and $\vb$, and $d(\va,\vb)$ denotes the Euclidean distance between the two points. For a point set $S\subset E^d$, define $d(\va,S):=\inf_{\vy\in S}d(\va,\vy)$. For another point set $X\subset E^d$, further define $\Delta(X,S):=\sup_{\vx\in X}d(\vx,S)$.

An integer point $\va=(a_1,\ldots,a_d)\in \bbZ^d$ is $primitive$, if $\gcd(a_1,\ldots,a_d)=1$. Geometrically, a primitive point is an integer point visible from the origin in the integer lattice $\bbZ^d$.
Given a positive real number $R$, denote $\N(R)$ the set of all primitive points of length not larger than $R$, i.e.,
\[\N(R)=\{\va\in \bbZ^d: \va \text{ is primitive, }\|\va\|\leq R\}.\]
Note that $\va\in \N(R)$ if and only if $-\va\in \N(R)$.

For each $\va=(a_1,\ldots,a_d)\in\N(R)$, denote $\P(\va)$ the family of all $(d-1)$-hyperplanes with normal vector $\va$ (i.e., hyperplanes orthogonal to $\va$), that is, $\{(X_1,\ldots,X_d)\in E^d:a_1X_1+\ldots+a_dX_d=b\}$ with $b\in \bbR$. Further, we can define a lattice $L(\va)=\{(x_1,\ldots,x_d)\in \bbZ^d:a_1x_1+\ldots+a_dx_d=0\}$, which is a $(d-1)$-sublattice of $\bbZ^d$. The maps $\P$ and $L$ are both injective in terms of pairs $(\va,-\va)$, that is, over $\N(R)$, we have
 \[(\va,-\va)\stackrel{1:1}{\longleftrightarrow}\P(\va)\stackrel{1:1}{\longleftrightarrow}L(\va).\]
So we say that each element of $\P(\va)$ is a hyperplane of $\va$ and each integer lattice translated from $L(\va)$ is a lattice of $\va$.
$~$

 Let $\lambda \in \bbZ$ be a positive integer. Define
 \begin{align*}
\N'_{\lambda}(R)=\{(a_1,\ldots,a_d)\in \N(R):&\gcd(\lambda^d,a_i)=\lambda^{d-i}~(3\leq i\leq d);\\
&\gcd(a_1,a_2)=\lambda^{d-2},\sqrt{a_1^2+a_2^2}\geq \lambda^{d-1}\}.
\end{align*}
  For example, when $d=3, R=6, \lambda=2$, we have $\N'_{2}(6)=\{(\pm4,\pm2,\pm1),(\pm2,\pm4,\pm1),$ $(\pm4,\pm2,\pm3),$ $(\pm2,\pm4,\pm3)\}$. Let $\mathcal{S}_d$ denote the set of permutations on $[d]$. Define
\[\N_\lambda(R)=\cup_{\tau\in \mathcal{S}_d}\left\{(a_{\tau(1)},\ldots,a_{\tau(d)}): (a_1,\ldots,a_d)\in\N'_{\lambda}(R)\right\}.\]
That is, $\N_\lambda(R)$ is the collection of all rearrangements of vectors in $\N'_{\lambda}(R)$.
Note that $\N_\lambda(R)\subset\N(R)$ and if $\va\in\N_{\lambda}(R)$, then $-\va\in\N_{\lambda}(R)$.

  %

Now we set the two  important parameters in $\N_\lambda(R)$. For fixed $d\geq 3$  and sufficiently large $n$, let
\begin{equation}\label{setRlambd}
  \text{$R:=\sqrt{d}n^\frac{d-1}{d+1}$ and $\lambda$ be the largest prime not more than $n^\frac{1}{d+1}$.}
\end{equation}
By \cite[Theorem 1]{baker2001difference}, $\lambda=(1-o(1))n^\frac{1}{d+1}$.
With the defined values of $R$ and $\lambda$, we will show that  $\N_\lambda(R)$ has the following two properties.

\begin{itemize}
  \item  For each primitive point  $\va\in \N_\lambda(R)$, {the lattice $L(\va)$ has no short vectors.} See \Cref{lattice distance}.
  \item As a point set in $E^d$, $\N_\lambda(R)$ is densely distributed around  the $d$-sphere of radius $R$ centered at the origin. See \Cref{dense distribution}.  Note that we describe its density in such a way: for any line through the origin, we can always find some point in $\N_\lambda(R)$ around the $d$-sphere that is close to the line.

\end{itemize}
These two properties are crucial for the proof of \Cref{g_i(H)}, which shows that the values of $g_i$ match with the desired domain of $f_i$, so as to bound the degree of $f_i$.

\vspace{0.5cm}
\begin{lemma}\label{lattice distance}
For any $\va\in \N_\lambda(R)$, {each vector in the lattice $L(\va)$ is of length at least $\lambda$.}
\end{lemma}
\begin{proof}
It suffices to show that the conclusion holds for any $\va=(a_1,\ldots,a_d)\in \N'_{\lambda}(R)$, that is, $\|\vx\|\geq\lambda$ for each nonzero integer point $\vx=(x_1,\ldots,x_d)\in L(\va)$ with $\va\in \N'_{\lambda}(R)$.


Let $k=\max\{i\in [d]:x_i\neq 0\}$. If $k\geq 3$, by the definition of $\N'_{\lambda}(R)$, $\gcd(a_1,\ldots,a_{k-1})=\lambda^{d-k+1}$. Since $a_1x_1+\cdots+a_kx_k=0$, we have $\lambda^{d-k+1}|a_kx_k$. Then $\gcd(\lambda^d,a_k)=\lambda^{d-k}$ implies  $\lambda|x_k$, and hence $\|\vx\|\geq |x_k|\geq \lambda$. If $k\leq 2$, we can write $a_1=a'_1\lambda^{d-2},a_2=a'_2\lambda^{d-2}$ with $a'_1,a'_2\in \bbZ\backslash\{0\}$ and $\gcd(a'_1,a'_2)=1$ by definition. Then $a'_1x_1+a'_2x_2=0$, which yields $a'_1|x_2,a'_2|x_1,x_1,x_2\neq 0$, and thus $|x_2|\geq |a'_1|$, $|x_1|\geq |a'_2|$. So $\|\vx\|\geq \sqrt{x_1^2+x_2^2}\geq \sqrt{{a'_2}^2+{a'_1}^2}\geq \lambda$. This completes the proof.
\end{proof}

\vspace{0.5cm}
\begin{lemma}\label{dense distribution} Fix an arbitrary integer $d\geq 3$. When $n$ is sufficiently large, for any line $l\in E^d$ through the origin, there exists some point $\vb\in\N_\lambda(R)$ such that
\[\|\vb\|=(1-o(1))R~~\text{and}~~d(\vb,l)\leq(1+o(1))\lambda^{d-2}.\]
\end{lemma}
\begin{proof}
Take a point $\va_1\in l$ such that $\|\va_1\|=(1-\epsilon)R$, where $\epsilon=\frac{2\lambda^{d-2}}{R}=o(1)$. Without loss of generality, let $\va_1=(x_1,\ldots,x_d)$ with $x_1\geq x_2\geq\cdots\geq x_d\geq 0$. Then $\sqrt{x_1^2+x_2^2}\geq\sqrt{\frac{2}{d}}(1-\epsilon)R$. We want to find some point $\vb$ in $\N_\lambda(R)$ close to $\va_1$.

Consider a map $\pi:E^d\rightarrow E^2$, which maps
\[\vr=(r_1,\ldots,r_d)\mapsto\overline{\vr}=(\overline{r}_1,\overline{r}_2)=\lambda^{-(d-2)}(r_1,r_2),\]
then $\pi$ is a projection onto the first two coordinates with certain scaling. For $\overline{\va}_1=\pi(\va_1)=(\overline{x}_1,\overline{x}_2)=\lambda^{-(d-2)}(x_1,x_2)$, recall that $R=\sqrt{d}n^\frac{d-1}{d+1}$ and $\lambda=(1-o(1))n^\frac{1}{d+1}$, we have
\[\|\overline{\va}_1\|=\sqrt{\overline{x}_1^2+\overline{x}_2^2}=\lambda^{-(d-2)}\sqrt{x_1^2+x_2^2}\geq\lambda^{-(d-2)}\sqrt{\frac{2}{d}}(1-\epsilon)R
\geq(1-o(1))\sqrt{2}\lambda.\]
Since $x_1\geq x_2\geq 0$, we have $\overline{x}_1\geq\overline{x}_2\geq 0$. Then $\overline{x}_1\geq\frac{\sqrt{2}}{2}\sqrt{\overline{x}_1^2+\overline{x}_2^2}=\Omega(\lambda)=\Omega(n^\frac{1}{d+1})$, which is large. Let $p$ be the largest prime satisfying $p\leq\overline{x}_1$. By \cite[Theorem 1]{baker2001difference}, $p=(1-o(1))\overline{x}_1$. Let $\overline{\va}=(p,q)\in \vo\overline{\va}_1\subset \pi(l)$. Then $q/p=\overline{x}_2/\overline{x}_1$ and thus $0\leq q\leq p$
and $\|\overline{\va}\|=\frac{p}{\overline{x}_1}\|\overline{\va}_1\|=(1-o(1))\|\overline{\va}_1\|$. Since $p$ is prime, $(p,i)$ is a primitive point in $\bbZ^2$ for all $i\in[p-1]$. Hence there exists a primitive point $\overline{\vb}=(p,i_0)$ such that $i_0\in[p-1]$ and $|i_0-q|\leq 1$. 
{Then
\[\|\overline{\vb}\|=\sqrt{p^2+i_0^2}=(1\pm o(1))\sqrt{p^2+q^2}=(1\pm o(1))\|\overline{\va}\|=(1\pm o(1))\|\overline{\va}_1\|\geq (1-o(1))\sqrt{2}\lambda.\]}
Let $\va=(y_1,\ldots,y_d)=\pi^{-1}(\overline{\va})\cap l$ be a preimage of $\overline{\va}$, then $y_1=\lambda^{d-2}p, y_2=\lambda^{d-2}q, \|\va\|/\|\va_1\|=\|\overline{\va}\|/\|\overline{\va}_1\|$ and thus \[\|\va\|=(1-o(1))\|\va_1\|=(1-o(1))(1-\epsilon)R.\]
For $\overline{\vb}=(p,i_0)$, take one of its preimage $\vb=(b_1,\ldots,b_d)\in\bbZ^d$ with $b_1=p\lambda^{d-2},b_2=i_0\lambda^{d-2}$ and $b_i=(\lfloor y_i/\lambda^{d-i}\rfloor+\delta_i)\lambda^{d-i}~(3\leq i\leq d)$, where $\delta_i\in \{0,1\}$ such that $\lambda\nmid (\lfloor y_i/\lambda^{d-i}\rfloor+\delta_i)$. In particular, $\lambda\nmid b_d$. Hence $b_1=y_1$, $|b_2-y_2|=\lambda^{d-2}|i_0-q|\leq\lambda^{d-2}$, $|b_i-y_i|\leq\lambda^{d-i}$ for $3    \leq i\leq d$ and thus \[d(\va,\vb)=\sqrt{(y_1-b_1)^2+(y_2-b_2)^2+\cdots+(y_d-b_d)^2}\leq (1+o(1))\lambda^{d-2}=o(R).\]
So we have $d(\vb,l)\leq d(\va,\vb)\leq(1+o(1))\lambda^{d-2}$. That is, we have found a point $\vb$ which is close to the line $l$.

Next, we need to show that $\vb\in\N_\lambda(R)$. Indeed, observe that
\begin{itemize}
  \item $\|\vb\|\leq \|\va\|+d(\va,\vb)\leq(1-o(1))(1-\epsilon)R+(1+o(1))\lambda^{d-2}\\
~~~~=(1-o(1))(R-2\lambda^{d-2})+(1+o(1))\lambda^{d-2}\leq(1-o(1))R\leq R$;
  \item $\gcd(b_1,b_2,\ldots,b_d)\leq \gcd(b_1,b_2,b_d)=\gcd(p\lambda^{d-2},i_0\lambda^{d-2},b_d)=\gcd(p,i_0,b_d)=1$;
  \item $\gcd(\lambda^d,b_i)=\lambda^{d-i}$ for $3\leq i\leq d$;
  \item $\gcd(b_1,b_2)=\lambda^{d-2}\gcd(p,i_0)=\lambda^{d-2}$; and
  \item
        {$\sqrt{b_1^2+b_2^2}=\lambda^{d-2}\sqrt{p^2+i_0^2}=\lambda^{d-2}\|\overline{\vb}\|\geq\lambda^{d-2}(1-o(1))\sqrt{2}\lambda>\lambda^{d-1}$.}
\end{itemize}
By definition, $\vb\in\N'_\lambda(R)\subset\N_\lambda(R)$.

Finally, we claim that $\|\vb\|=(1-o(1))R$. We see that $\|\vb\|\leq(1-o(1))R$ above, and similarly, $\|\vb\|\geq \|\va\|-d(\va,\vb)=(1-o(1))R-o(R)=(1-o(1))R$. Hence $\|\vb\|=(1-o(1))R$. This completes the proof.
\end{proof}

\subsection{Idea of the polynomial construction}


Given an arbitrary nonempty set $H\subset [n]^d$, we briefly explain the idea of the construction of our polynomial  $p=p_0p_1\cdots p_{d-1}$ with each $p_i=f_i\circ g_i$, where $f_i$ is a univariate polynomial provided by \Cref{f1} or \Cref{f2}, and $g_i$ is a $d$-variable direction function. Required by \Cref{d_polynomial}, we aim to construct a polynomial $p$ with a peak value and of degree $O(n^\frac{d}{d+1})$. Observing that $\deg p=\sum\deg p_i=\sum\deg f_i$ as $g_i$ is linear, we need to bound the degree of $f_i$ up to $O(n^\frac{d}{d+1})$.
{Required by property (c) of \Cref{f1} and \Cref{f2},} we need to bound the range of values of $g_i$ on $H$.  For convenience, let $g_i(X)$ be the set of values of $g_i$ on all points in $X$ for some  $X\subset E^d$.

In the construction, we take $f_0$ from \Cref{f1} with the domain $[0,b_0]$ and $f_0(0)=m_0$, and $f_i$ from \Cref{f2} with the domain $[-a,b]$  and $f_i(0)=m$ for $1\leq i\leq d-1$, where $b_0=O(n^\frac{2d}{d+1})$, $a=O(n^\frac{d-1}{d+1})$, $b=\sqrt{d}n$ and $m_0, m$ are two positive constants. The explicit definitions of these parameters are given in {Eq. (\ref{eq-f1}) and Eq. (\ref{eq:ab})}. So we can apply  \Cref{f1} (c) and \Cref{f2} (c) to get $\deg f_i=O(n^\frac{d}{d+1})$ for $0\leq i\leq d-1$.

For the $d$ direction functions $g_i$, let $\va_i$ be their direction vectors and $\alpha_i$ be their zero-valued $(d-1)$-hyperplanes, respectively, i.e., $g_i(\alpha_i)=0$. In the construction, we take the point $\vh\in H$ in \Cref{d_polynomial} to be the unique intersecting point of these $d$ zero-valued hyperplanes $\alpha_i$, i.e., $\vh=\alpha_0\cap\cdots\cap\alpha_{d-1}$. Then define $g_i(\vx)=\va_i\cdot(\vx-\vh)$. This requires that the $d$ direction vectors $\va_i$ are linearly independent.

{
In order to reach the conclusion of  \Cref{d_polynomial},} we need each $g_i$ possessing the following three properties.

\vspace{0.3cm}
\textbf{Property A}: $g_0(H)\subset[0,b_0]$ and $g_i(H)\subset[-a,b],~1\leq i\leq d-1$.

\vspace{0.3cm}
\textbf{Property B}: The map $(g_0,g_1,\ldots,g_{d-1}):E^d\rightarrow E^d, \vx\mapsto (g_0(\vx),g_1(\vx),\ldots,g_{d-1}(\vx))$ is injective.

\vspace{0.3cm}
\textbf{Property C}: $g_0(H)\subset\bbZ$ and $g_i(H\cap g_0^{-1}(k))\subset\bbZ+\epsilon_{ik}$ with a constant $\epsilon_{ik}\in(-\frac{1}{2},\frac{1}{2}]$ for $ 1\leq i\leq d-1,k\in \bbZ$. That is, $g_i(H\cap g_0^{-1}(k))$ is contained in a translation of $\bbZ$.

\vspace{0.3cm}
Note that \textbf{Property A} is used to bound the degrees of $f_i$ and hence guarantees that the polynomial $p$ is of degree $O(n^\frac{d}{d+1})$. \textbf{Property B} is equivalent to that the $d$ direction vectors $\va_i$ are linearly independent and thus the $d$ zero-valued hyperplanes $\alpha_i$ intersect at a unique point $\vh$. \textbf{Property C} will be used to show that the sums $\sum_{g_0(H)}|f_0|$ and $\sum_{g_i(H\cap g_0^{-1}(k))}|f_i|$ are upper bounded when applying \Cref{f1} (b) and \Cref{f2} (b). Therefore, combining \textbf{Property B} with \textbf{Property C}, we are able to estimate the following sum as in Eq. (\ref{eq-p2})
\begin{align}\label{eq-peak}
\sum_{\vx\in H}|p(&\vx)|=\sum_{\vx\in H}|f_0\circ g_0(\vx)||f_1\circ g_1(\vx)|\cdots|f_{d-1}\circ g_{d-1}(\vx)|\notag\\
&\overset{\textbf{Property~B}}{\leq} \sum_{k\in g_0(H)}|f_0(k)|(\sum_{z_1\in g_1(H\cap g_0^{-1}(k))}|f_1(z_1)|)\cdots(\sum_{z_{d-1}\in g_{d-1}(H\cap g_0^{-1}(k))}|f_{d-1}(z_{d-1})|)\notag\\
&{\overset{\textbf{Property~C}}{<} 2p(\vh).}
\end{align}
This deduction degenerates the polynomial $p$ into univariate case so that one can construct it by applying the polynomials $f_i$ for sequence reconstruction. Note that in Eq. (\ref{eq-peak}), the first inequality results immediately from \textbf{Property B}, and the second inequality holds after careful computations under \textbf{Property C}, see details in \Cref{p(H)}. Also note that Eq. (\ref{eq-peak}) implies Eq. (\ref{equation d}). Therefore, \textbf{Properties A-C} of $g_i$ imply that $p$ has  degree $O(n^\frac{d}{d+1})$ and has a peak value $p(\vh)$, thus verify \Cref{d_polynomial}.

\vspace{0.5cm}
So the main task next is to construct direction functions ${g_i}$ so that they satisfy \textbf{Properties A-C}.


Given $f_0$ from \Cref{f1} and $f_i$ for $1\leq i\leq d-1$ from \Cref{f2} above, we will take the direction vector $\va_0$ from $\N_\lambda(R)$ and the $d-1$ direction vectors $\va_i$, $1\leq i\leq d-1$, being unit vectors. Then for any $ \vx\in E^d$, $g_0(\vx+\frac{\va_0}{R})-g_0(\vx)\leq 1$ and $g_i(\vx+\va_i)-g_i(\vx)=1,1\leq i\leq d-1$. Hence we have
\[\max_{H}g_0-\min_{H}g_0\leq O(nR)=O(n^\frac{2d}{d+1})\leq b_0,~\max_{H}g_i-\min_{H}g_i\leq O(n)\leq b+a,1\leq i\leq d-1,\]
required by \textbf{Property A}.

Given direction vectors $\va_i$ above, $g_i$ is determined by its zero-valued hyperplane $\alpha_i$. Geometrically, \textbf{Property A} implies that $\alpha_0$ \emph{is tangent to} $H$ (see \Cref{tangent}) as $g_0(H)\geq 0$, and for $1\leq i\leq d-1$, $\alpha_i$ is approximately tangent to $H$ as 
{$\min g_i(H)\geq -a=-O(n^\frac{d-1}{d+1})=-o(n)$}. Combined with that $\vh=\alpha_0\cap\cdots\cap\alpha_{d-1}$, we see $\alpha_i$ is a perturbation of $\alpha_0$ around $\vh$ and thus $\va_i$ is a unitized perturbation of $\va_0$ for $1\leq i\leq d-1$. Then our construction follows three steps as below. Here, all deductions are roughly explained and details will be given in the later sections.

\vspace{0.3cm}
\textbf{Step 1.} Take
$S$  a $d$-sphere of radius $\frac{\sqrt{d}}{2}n$ in $E^d$, and
$P$ a circumscribed hyperpolygon of $S$ whose
facets formed by two hyperplanes in $\P(\va)$ that are tangent to $S$, where $\va$ runs over $\N_\lambda(R)$. We translate $S\cup P$ carefully so that
 $P$ \emph{is tangent to} $H$, that is, all points in $H$ are inside of $P$ and some points in $H$ touches $P$.
See \Cref{Fig.1}.

\begin{figure}[H]
\centering 
\begin{minipage}[h]{0.45\textwidth} 
\centering 
\includegraphics[width=1\textwidth]{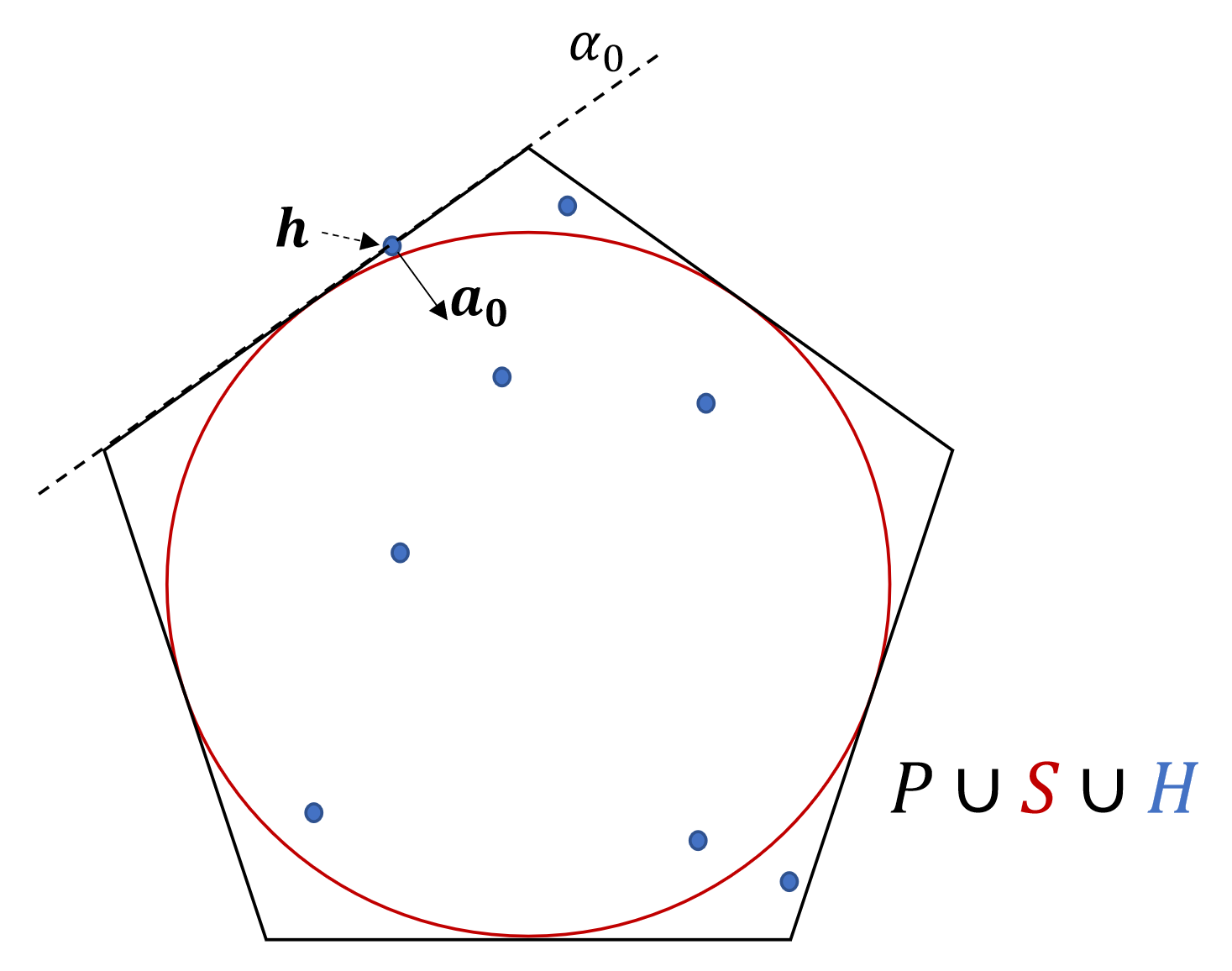} 
\caption{$P$ is tangent to $H$}
\label{Fig.1}
\end{minipage}
\begin{minipage}[h]{0.45\textwidth} 
\centering 
\includegraphics[width=1\textwidth]{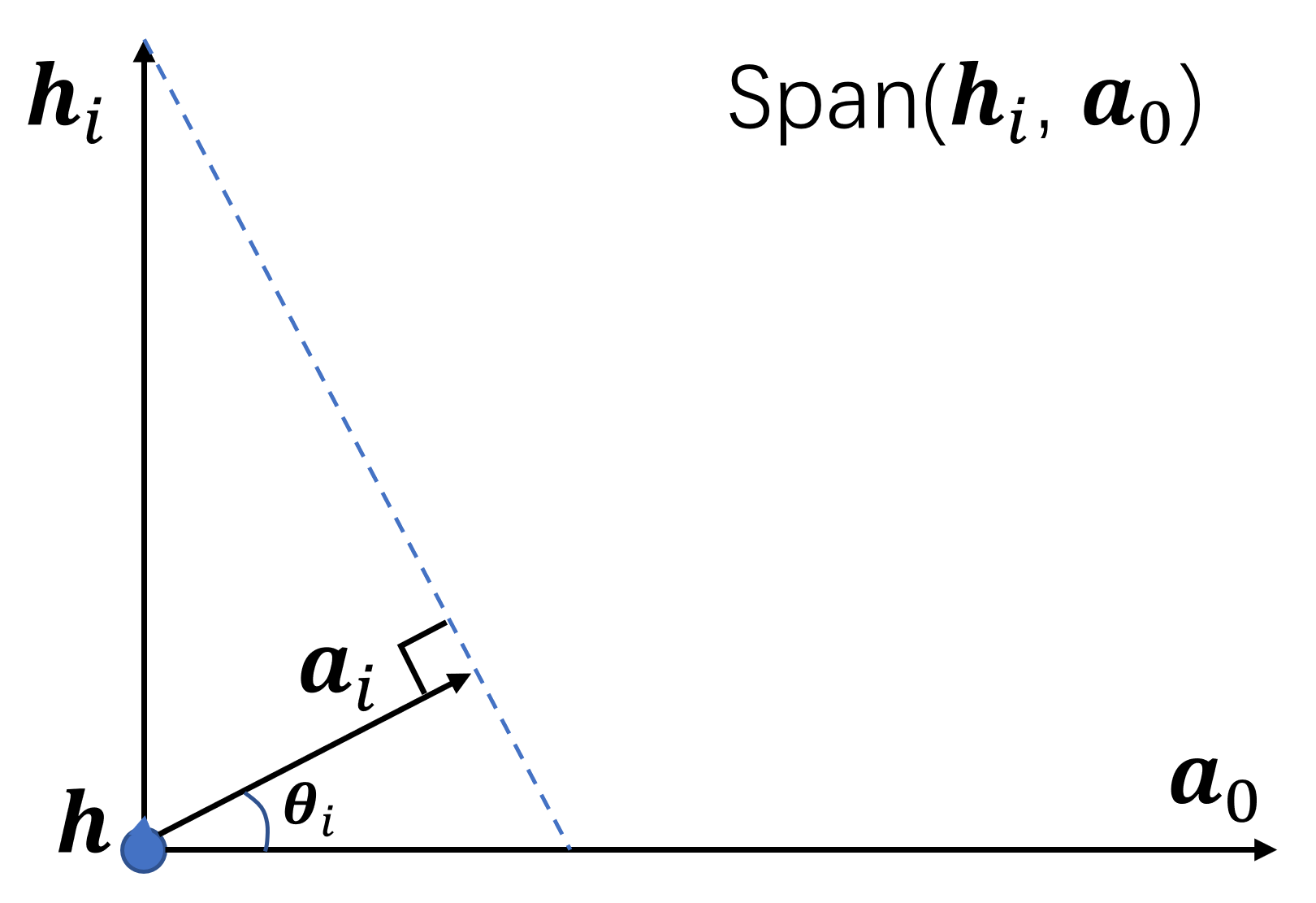}
\caption{Construct $\va_i$ in $\text{Span}(\vh_i,\va_0)$}
\label{Fig.2}
\end{minipage}
\end{figure}

Take $\alpha_0$ a facet of $P$ that intersects $H$, and take $\vh\in \alpha_0\cap H$. Then
$\alpha_0$ is associated with a hyperplane (still denoted by $\alpha_0$) of some $\va_0\in \N_\lambda(R)$ and $\alpha_0$ is tangent to $H$. Construct the first direction function $g_0(\vx)=\va_0\cdot(\vx-\vh)$. Then $g_0$ satisfies
\textbf{Property A} with $g_0(H)\subset[0,\sqrt{d}nR]$ and \textbf{Property C} with $g_0(H)\subset\bbZ$; see \Cref{g}.

Further, this careful construction of $S\cup P$ forces all facets of $P$ are small and thus $P$ is close to $S$; see \Cref{DPS}. This is because $\N_\lambda(R)$ is densely distributed around the $d$-sphere of radius $R$ centered at the origin, i.e., \Cref{dense distribution}. Then the fact that $P$ is tangent to $H$ implies that $S$ is approximately tangent to $H$.
This is helpful for the later constructions of $g_i$ whose zero-valued hyperplane $\alpha_i$ is approximately tangent to $S$.

\vspace{0.3cm}
\textbf{Step 2.} Fix the above point $\vh$ as the origin of $E^d$, and take a Mahler  basis $\vy_1,\ldots,\vy_{d-1}$ of the lattice $L(\va_0)$.  Then
$L(\va_0)$ is the maximal integral $(d-1)$-lattice in $\alpha_0$ as $\vh\in \alpha_0$ is the origin. Since $\va_0\in \N_\lambda(R)$, $\|\vy_i\|\geq \lambda$ by \Cref{lattice distance}. For $1\leq i\leq d-1$, let the \emph{height vector} of $\vy_i$ be $\vh_i$,  that is the projection of $\vy_i$ in the direction orthogonal to all $\vy_j$ with $j\neq i$, see \Cref{defhv}. By the approximate orthogonality of Mahler basis \cite[Corollary 3.35]{tao2006additive}, we have
 \[\|\vh_i\|\approx \|\vy_i\| =\Omega(\lambda)=\Omega(n^\frac{1}{d+1}),1\leq i\leq d-1.\]
See \Cref{h_i}.
\vspace{0.3cm}

\textbf{Step 3.} As shown in \Cref{Fig.2}, construct each unit vector $\va_i$ in the plane $\text{Span}(\vh_i,\va_0)$ by rotating $\va_0$ towards $\vh_i$ with a small angle $\theta_i=\arcsin\|\vh_i\|^{-1}$. Let direction functions $g_i(\vx)=\va_i\cdot(\vx-\vh),1\leq i\leq d-1$.
Denote $\alpha_i$ the hyperplane orthogonal to $\va_i$ through $\vh$, then $g_i(\alpha_i)=0$. Since $\va_i$ is a unitized perturbation of $\va_0$ around $\vh$, $\alpha_i$ is a perturbation of $\alpha_0$ around $\vh$ and $\vh\in \alpha_0\cap\cdots\cap\alpha_{d-1}$.
\begin{itemize}
  \item Since $\alpha_0$ is tangent to $S$, $\alpha_i$ is approximately tangent to $S$. This implies that  $g_i(H)\geq -o(n)$ (With careful computation, $g_i(H)\geq -O(n^\frac{d-1}{d+1})$).
Hence $g_i(H)\subset[-a,b],~1\leq i\leq d-1$ (\textbf{Property A}); see \Cref{g_i(H)}.
\item Since $\vy_1,\ldots,\vy_{d-1}$ is a basis of the lattice $L(\va_0)$  and  $\va_i\in \text{Span}(\vh_i,\va_0)$ ($1\leq i\leq d-1$), $\va_i \perp \vy_j$ for any $i\neq j$.
Thus $g_i(L(\va_0))=g_i(\bbZ \vy_1+\cdots+\bbZ \vy_{d-1})=g_i(\bbZ \vy_i),1\leq i\leq d-1$, which implies that
 the map $(g_1,\ldots,g_{d-1}):E^d\rightarrow E^{d-1},\vx\mapsto (g_1(\vx),\ldots,g_{d-1}(\vx))$ is injective restricted on $L(\va_0)=g_0^{-1}(0)\cap \bbZ$. So \textbf{Property B} is obtained; see \Cref{rem}.
\item By definition of height vectors, the projection of the lattice $L(\va_0)$ on $\text{Span}(\vh_i,\va_0)$ degenerates into a $1$-lattice with the basis $\vh_i$. Observing that the projection of $\vh_i$ on $\text{Span}(\va_i)$ is $\va_i$, we have $g_i(L(\va_0))=\bbZ$.
 So $g_i(H \cap L(\va_0))=g_i(H\cap g_0^{-1}(0))\subset\bbZ$, that is, \textbf{Property C} holds for $k=0$. By the periodicity of the lattice $\bbZ^d$,
 $g_i(H\cap g_0^{-1}(k))\subset\bbZ+\epsilon_{ik},1\leq i\leq d-1,k\in \bbZ$, where $\epsilon_{ik}\in(-\frac{1}{2},\frac{1}{2}]$ is a constant. So \textbf{Property C} holds for any $k$.
\end{itemize}


\vspace{0.5cm}
\section{Construction of the polynomial}\label{sec:con}

Recall that  the polynomial $p$ is of the form $p=p_0p_1\cdots p_{d-1}$. Next we construct $p_i$ for different $i$ separately.

\subsection{Construction of $p_0=f_0\circ g_0$}

\begin{definition}\label{tangent}
For a bounded point set $H\subset E^d$, a $(d-1)$-hyperplane $\alpha$ \emph{is tangent to} $H$, if $\alpha\cap H\neq\emptyset$ and $H$ lies on one side of $\alpha$. Similarly, a $d$-sphere $S$ \emph{is tangent to} $H$, if $S\cap H\neq\emptyset$ and $H$ lies inside $S$.


\end{definition}
\vspace{0.5cm}

Now we construct the hyperpolygon $P$ and the $d$-sphere $S$ in \Cref{Fig.1} as follows.

Given an arbitrary nonempty set $H\subset [n]^d$, let $S'$ be a $d$-sphere of radius $\frac{\sqrt{d}}{2}n$ such that $H$ is inside of $S'$. This is allowable since we can put all points of $[n]^d$ inside $S'$. For each $\va\in\N_\lambda(R)$, $\P(\va)$ contains exactly two hyperplanes tangent to $S'$, corresponding to the pair  $(\va,-\va)$ as their normal vectors, respectively. Let $\va$ go through $\N_\lambda(R)$, whose elements are all primitive points. Then all pairs of  hyperplanes  tangent to $S'$ form a  circumscribed convex hyperpolygon $P'$ of $S'$, and each facet of $P'$ one-to-one corresponds to an element of $\N_\lambda(R)$. It is clear that $H$ is also inside $P'$.

Translate $P'\cup S'$ to $P\cup S$ such that $P$ is tangent to $H$, that is, some facet $\alpha_0$ of $P$ from some $\P(\va_0)$ is tangent to $H$ and $H$ is still inside $P$, where $\va_0$ points to the same side of $\alpha_0$ as $H$. Pick some point $\vh\in \alpha_0\cap H$, then $\vh$ will be the desired point in \Cref{d_polynomial}.

\vspace{0.3cm}
Note that after this translation, some points in $H$ may be outside $S$ as in \Cref{Fig.1}. Especially for the selected point $\vh\in \alpha_0\cap H\subset P$, since $P$ is outside $S$, $\vh$ might be outside $S$ and generally $\vh\notin S$. However, we claim that  $S$ is indeed approximately tangent to $H$, that is, all points of $H$ outside $S$ are close to $S$. Since $P$ is tangent to $H$, it suffices to show that $P$ is close to $S$, or mathematically, $\Delta(P,S)$ is small. 
 Indeed, the following lemma tells us that $\Delta(P,S)=O(n^\frac{d-1}{d+1})$ which is small comparing to the radius $\frac{\sqrt{d}}{2}n$ of $S$.

\begin{figure}[H]
\centering
\includegraphics[width=0.5\textwidth]{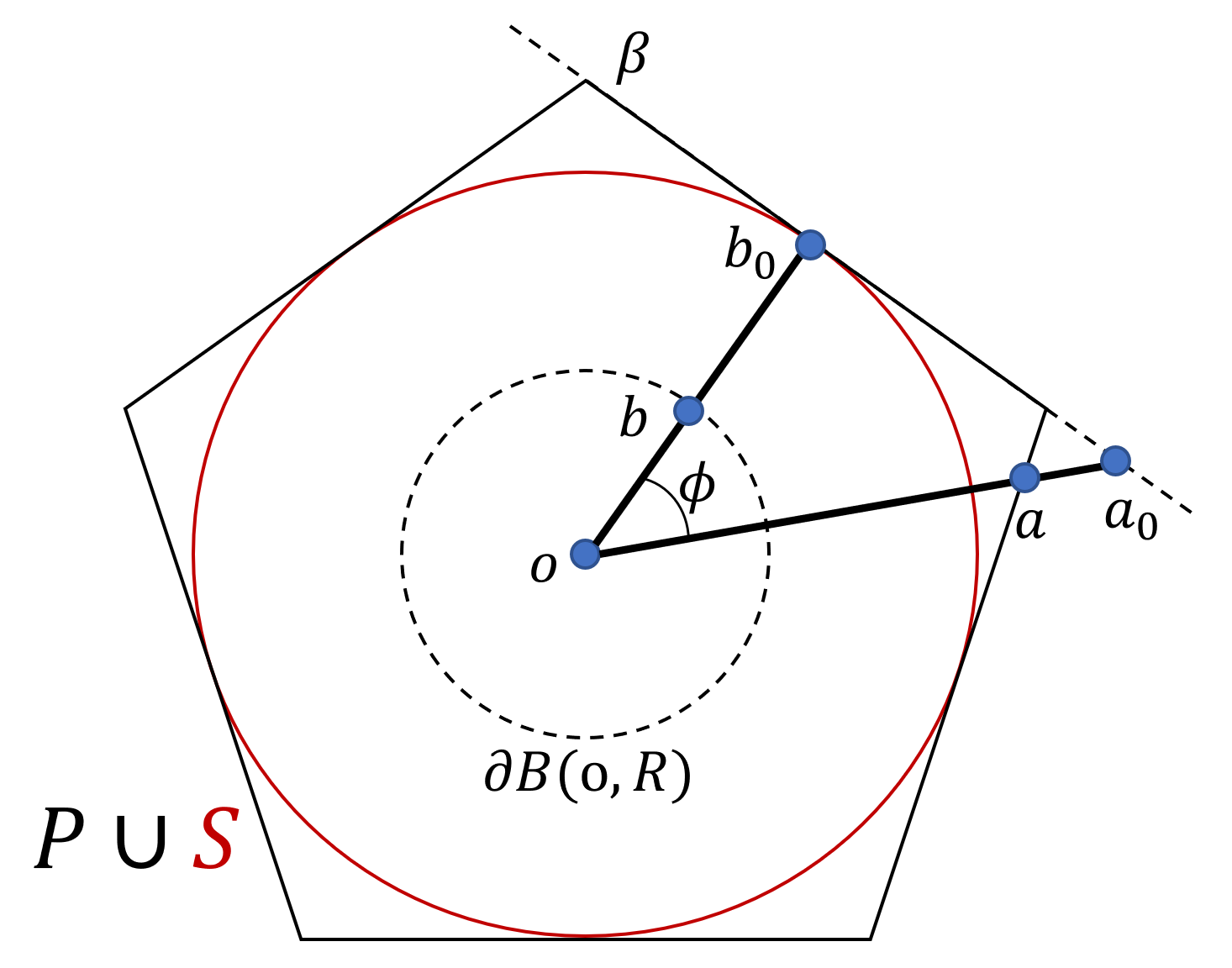}
\caption{For \Cref{DPS}. Here $\partial B(\vo,R)$ denotes the $d$-sphere of radius $R$.}
\label{Fig.3}
\end{figure}

\begin{lemma}\label{DPS}
For fixed $d\geq 3$, when $n$ is sufficiently large, 
\[\Delta(P,S)\leq (1+o(1))\frac{1}{4\sqrt{d}}n^\frac{d-1}{d+1}.\]
\end{lemma}
\begin{proof}
Assume that the center of the $d$-sphere $S$ is the origin $\vo$. As shown in \Cref{Fig.3}, since the hyperpolygon $P$ is a compact point set, there exists some point $\va\in P$ such that $d(\va,S)=\Delta(P,S)$. On the other hand, $d(\va,S)=\|\va\|-\frac{\sqrt{d}}{2}n$ as $\va\in P$ is outside $S$. Extend $\vo\va$ to get a line $l$ through the origin $\vo$. By \Cref{dense distribution}, there exists some point $\vb\in\N_\lambda(R)$ such that $\|\vb\|=(1-o(1))R$ and $d(\vb,l)\leq(1+o(1))\lambda^{d-2}$. By replacing $\vb$ with $-\vb$ if necessary, we have $d(\vb,\vo\va)=\min_{\vx\in \vo\va} d(\vb,\vx)\leq(1+o(1))\lambda^{d-2}$, that is, $\vb$ is close to the segment $\vo\va$.

Extend $\vo\vb$ to intersect $S$ at point $\vb_0$, then $\|\vb_0\|=\frac{\sqrt{d}}{2}n$. Since $\vb\in\N_\lambda(R)$, $P$ has a surface $\beta\in\P(\vb)$ such that $\vb_0=\beta\cap S$, i.e., $\vb_0$ is the tangent point of $\beta$ and $S$. By the convexity of $P$, $\va$ and $\vo$ are on the same side of $\beta$. Extend $\vo\va$ to intersect $\beta$ at point $\va_0$. Let $\phi=\angle(\va,\vb)=\angle(\va_0,\vb_0)$, then $\sin\phi=\frac{d(\vb,l)}{\|\vb\|}\leq(1+o(1))\frac{\lambda^{d-2}}{R}=O(n^{-\frac{1}{d+1}})=o(1)$ and $\cos\phi=1-o(1)$. This implies that $\phi$ is small and thus $\va$ is closer to $\beta$ than $\vo$. Then $\va\in \vo\va_0$, $\|\va_0\|\geq \|\va\|$ and
\[\Delta(P,S)=d(\va,S)=\|\va\|-\frac{\sqrt{d}}{2}n\leq \|\va_0\|-\frac{\sqrt{d}}{2}n.\]
It suffices to compute the value of $\|\va_0\|$.

Since $\vb_0$ is the tangent point of $\beta$ and $S$, we have $\angle(\va_0-\vb_0,\vo-\vb_0)=\frac{\pi}{2}$ and thus three points $\vo,\va_0,\vb_0$ form a right triangle. Then
\begin{align*}
\Delta(P,S)&\leq \|\va_0\|-\frac{\sqrt{d}}{2}n=\frac{\|\vb_0\|}{\cos\phi}-\frac{\sqrt{d}}{2}n=\frac{\sqrt{d}}{2}n(\frac{1}{\cos\phi}-1)=\frac{\sqrt{d}}{2}n\frac{\sin^2\phi}{\cos\phi(1+\cos\phi)}\\
&\leq(1+o(1))\frac{\sqrt{d}}{4}n\left((1+o(1))\frac{\lambda^{d-2}}{R}\right)^2\\
&=(1+o(1))\frac{1}{4\sqrt{d}}n^\frac{d-1}{d+1}.
\end{align*}
\end{proof}

\vspace{0.3cm}
Now we construct the first direction function
\[g_0(\vx):=\va_0\cdot(\vx-\vh)\]
 and a univariate polynomial $f_0$ provided by \Cref{f1} with
\begin{equation}\label{eq-f1}
  \text{domain }[0,b_0]:=[0,\sqrt{d}nR], \text{ and}~f_0(0)=m_0:=\frac{\pi^2}{6}4^d.
\end{equation}
Then compose $g_0$ into $f_0$ to get the first polynomial $p_0=f_0\circ g_0$. We see $g_0$ satisfies \textbf{Property A} and \textbf{Property C} as in the following lemma.

\vspace{0.5cm}
\begin{lemma}\label{g}
Given the direction function $g_0$ above, we have $g_0(H)\subset [0,\sqrt{d}nR)\cap\bbZ$.
\end{lemma}
\begin{proof} Let $\vx$ be an arbitrary point in $H$. First, since $\vh\in H \subset [n]^d$ and $\va_0\in \bbZ^d$, we have $g_0(\vx)=\va_0\cdot(\vx-\vh)\in \bbZ$. Second, since
$\va_0$ is the normal vector of $\alpha_0$ and points to the same side of $\alpha_0$ as $H$, then $\angle(\vx-\vh,\va_0)\leq \frac{\pi}{2}$ and $g_0(\vx)=\va_0\cdot(\vx-\vh)\geq 0$. Third, since the
diameter of $H$ is less than the diameter of $[n]^d$, which is less than $\sqrt{d}n$, then
 $\|\vx-\vh\|<\sqrt{d}n$ and $g_0(\vx)=\va_0\cdot(\vx-\vh)\leq \|\va_0\|\|\vx-\vh\|<\sqrt{d}nR$.
Thus we show that $g_0(H)\subset [0,\sqrt{d}nR)\cap\bbZ$.
\end{proof}

\subsection{Lattice and Mahler basis}

From now on, we fix $\vh\in \alpha_0\cap H$ as the origin of $E^d$ and assume that $\va_0$ begins from $\vh$. 
{Then $\alpha_0$ is a $(d-1)$-dimensional subspace of $E^d$ 
and thus the lattice $L(\va_0)\subset \alpha_0$.} By definition, $L(\va_0)$ is the collection of all integer points in $\alpha_0$, hence $L(\va_0)$ is the maximal integral $(d-1)$-lattice in $\alpha_0$.

Take a Mahler basis $\vy_1,\ldots,\vy_{d-1}$ of $L(\va_0)$ with respect to the unit ball \cite[Theorem 3.34]{tao2006additive}, where each $\vy_i$ begins from the origin $\vh$. Then $\vy_1,\ldots,\vy_{d-1}$ is also an integer basis of $\alpha_0$ as $L(\va_0)\subset \alpha_0$. By \cite[Corollary 3.35]{tao2006additive}, for $d\geq 3$,
\begin{align}\label{Mahler-basis}
{\prod_{i=1}^{d-1}\|\vy_i\|\leq\frac{2(d-1)!\Gamma(\frac{d-1}{2}+1)}{\pi^{(d-1)/2}}\det L(\va_0)\leq\frac{8\sqrt{2}}{81\pi}\frac{d^{3d/2}}{d^{1/2}(d-1)^{3/2}}\det L(\va_0):=\vartheta_d\det L(\va_0)}.
\end{align}

Recall that for $1\leq i\leq d-1$, the direction function $g_i(\vx)=\va_i\cdot(\vx-\vh)$ involves a unit direction vector $\va_i$. To construct $\va_i$, we need to introduce the ``height vector" of $\vy_i$ for $1\leq i\leq d-1$.
\vspace{0.5cm}\begin{definition}\label{defhv}
  We say $\vh_i$ is the \emph{height vector} of $\vy_i$ with respect to the basis $\vy_1,\ldots,\vy_{d-1}$ of $\alpha_0$, if $\vh_i$ is the projection of $\vy_i$ onto the orthogonal complementary space $\text{Span}(\vy_1,\ldots,\vy_{i-1},\vy_{i+1},\ldots,\vy_{d-1})^\bot$.
\end{definition}
\vspace{0.5cm}

Note that $\vh_i$  also begins from $\vh$. 
We call it height vector as we could think about the hypercube formed by $\vy_1,\ldots,\vy_{d-1}$.
For example, consider a parallelogram $\parallelogram BACD$ in $E^2$, then $\overrightarrow{AB},\overrightarrow{AC}$ form a basis. Let $\vh_C$ be the height vector of $\overrightarrow{AC}$ with respect to the basis. Obviously, $\|\vh_C\|$ is the height of $\parallelogram BACD$ with respect to the base side $AB$.

For these chosen height vectors $\vh_1,\ldots,\vh_{d-1}$, we expect that each $\vh_i$ is of large length for the following reason. Each $\va_i$ is supposed to be a unit vector obtained by perturbing $\va_0$ in the later construction. Since $\vh_i\bot \va_0$ (as $\vh_i\in \alpha_0$), if $\vh_i$ has small length, the projection of $\vh_i$ on $\text{Span}(\va_i)$ is small and thus the projection of $L(\va_0)$ on $\text{Span}(\va_i)$ is dense. Since each hyperplane orthogonal to $\va_i$ is a level set of $g_i$, $g_i(L(\va_0))$ as a point set in $\bbR$ is densely distributed and thus the sum $\sum_{\vx\in H\cap L(\va_0)}f_i\circ g_i(\vx)$ may be large. Then $p(\vh)$ may not be a peak value under the magnification of Eq. (\ref{eq-peak}). See details of arguments in \Cref{g_i} and \Cref{p slice}.

Indeed, each $\vh_i$ is of  length $\Omega(n^\frac{1}{d+1})$, which is confirmed by the following lemma.

\vspace{0.5cm}
\begin{lemma}\label{h_i}
Given the Mahler basis $\vy_1,\ldots,\vy_{d-1}$ of $L(\va_0)$ with height vectors $\vh_1,\ldots,\vh_{d-1}$ above, we have $\|\vh_i\|\geq \vartheta_d^{-1}\lambda$ for $1\leq i\leq d-1$.
\end{lemma}
\begin{proof} Consider the $(d-1)$-hypercube in $\alpha_0$ formed by
 $\vy_1,\ldots,\vy_{d-1}$:  
\[V_0=\{\lambda_1\vy_1+\lambda_2\vy_2+\cdots+\lambda_{d-1}\vy_{d-1}:\lambda_i\in [0,1],1\leq i\leq d-1\}.\]
For  $1\leq i\leq d-1$, let $V_i$ be the $(d-2)$-undersurface of $V_0$ with respect to $\vy_i$, i.e., $V_i=V_0\cap\text{Span}(\vy_1,\ldots,\vy_{i-1},\vy_{i+1},\ldots,\vy_{d-1})$. We easily deduce the following conclusions on the volume of $V_i$, $0\leq i\leq d-1$.
\begin{itemize}
  \item Since $\{\vy_1,\ldots,\vy_{d-1}\}$ is a basis of $L(\va_0)$, the hypervolume $|V_0|=\det L(\va_0)$.
  \item Since $V_i$ ($1\leq i\leq d-1$) is formed by $\vy_1,\ldots,\vy_{i-1},\vy_{i+1},\ldots,\vy_{d-1}$, the hyperarea $|V_i|\leq \prod_{j\neq i}\|\vy_j\|$.
\end{itemize} Since $\va_0\in\N_\lambda(R)$, by \Cref{lattice distance}, $\|\vy_i\|\geq\lambda$ for $1\leq i\leq d-1$.
Then by definition of height vector,
\[\|\vh_i\|=\frac{|V_0|}{|V_i|}\overset{}{=}\frac{\det L(\va_0)}{|V_i|}\overset{}{\geq}\frac{\det L(\va_0)\cdot\|\vy_i\|}{\prod_{j=1}^{d-1}\|\vy_j\|}\overset{}{\geq}\vartheta_d^{-1}\lambda,\]
where the last inequality applies Eq. (\ref{Mahler-basis}).
\end{proof}

\subsection{Construction of $p_i=f_i\circ g_i,1\leq i\leq d-1$}

Recall that $\vh$ is the origin and $\vy_1,\ldots,\vy_{d-1}$ is a Mahler basis of $L(\va_0)$ and thus a basis of $\alpha_0$. Then for $1\leq i\leq d-1$, $U_i:=\text{Span}(\vy_1,\ldots,\vy_{i-1},\vy_{i+1},\ldots,\vy_{d-1})$ is a {$(d-2)$-dimensional subspace} of $E^d$. By definition of height vector, $\vh_i\bot U_i$. Since $\va_0\bot\alpha_0$ and $\vh_i,U_i\subset\alpha_0$, we have $\va_0\bot\vh_i,\va_0\bot U_i$, and thus the $2$-dimensional orthogonal complementary space $U_i^\bot=\text{Span}(\va_0,\vh_i),1\leq i\leq d-1$.

Now we construct the direction vector $\va_i$ for $1\leq i\leq d-1$. Define
\[\theta_i:=\arcsin\|\vh_i\|^{-1}.\] Then the
unit vector $\va_i$ is obtained by rotating $\va_0$ towards $\vh_i$ with a small angle $\theta_i$ in the space $\text{Span}(\va_0,\vh_i)$, see \Cref{Fig.2}. Let \[g_i(\vx):=\va_i\cdot(\vx-\vh).\]

Now set
\begin{equation}\label{eq:ab}
  a:=\frac{10}{33}\sqrt{d}\vartheta_d^2n^\frac{d-1}{d+1},~b:=\sqrt{d}n \text{ and } m:=\frac{2}{3}\pi^2(d-1).
\end{equation}
For each $1\leq i\leq d-1$, let $f_i:=f$ be the same univariate polynomials provided by \Cref{f2} with
$\text{domain }[-a,b] $ and $f_i(0)=m$.
Then compose $g_i$ into $f_i$ to get the rest polynomials $p_i=f_i\circ g_i$.

\vspace{0.3cm}
Indeed, our construction above endows direction functions $g_i$ with the three required properties, see the proof of \textbf{Property A} in \Cref{g_i(H)},  \textbf{Property B} in \Cref{g_i} and  \textbf{Property C} in \Cref{g_i L_k}.

\begin{figure}[H]
\centering
\includegraphics[width=0.5\textwidth]{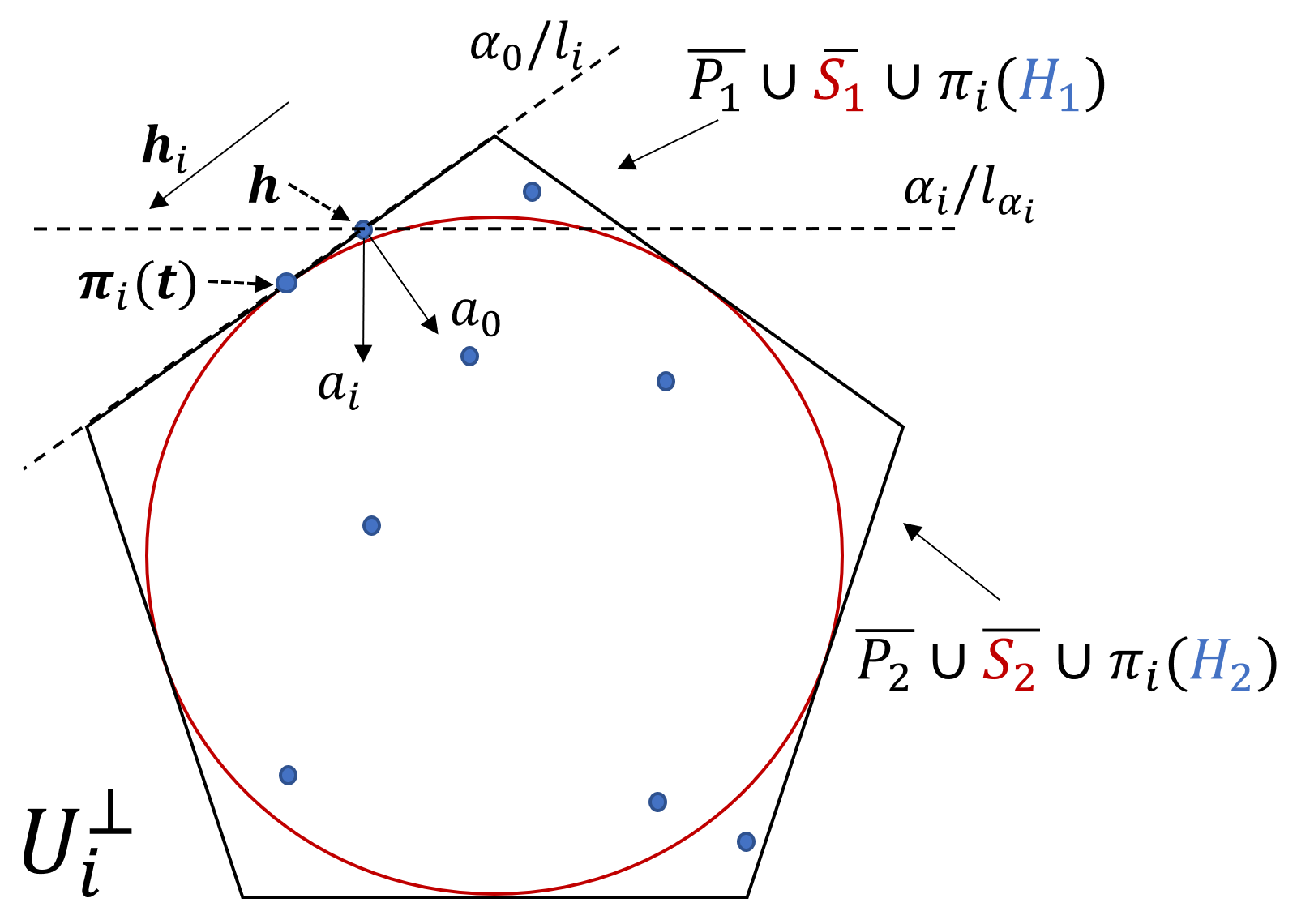}
\caption{\Cref{g_i(H)}}
\label{Fig.4}
\end{figure}

First, let us describe \textbf{Property A} of $g_i$ in a geometric point of view. Denote by $\alpha_i$ the $(d-1)$-hyperplane orthogonal to $\va_i$ and $\vh\in\alpha_i$. Then $g_i(\alpha_i)=0$. Since $\va_i$ is a unitized perturbation of $\va_0$, then $\alpha_i$ is a perturbation of $\alpha_0$. Define two subsets of $H$: 
\[H_1=\{\vx\in H:g_i(\vx)\leq 0\},\quad H_2=\{\vx\in H:g_i(\vx)\geq 0\}.\]
Then $H_1\cup H_2=H$ and  they are the two parts of $H$ partitioned by $\alpha_i$, where points of $H_2$ lie on the same side of $\va_i$  and points of $H_1$ lie on the other side, as shown in \Cref{Fig.4}. 
Recall that $\Delta(X,S)=\max_{\vx\in X}d(\vx,S)$ for any two nonempty sets $X,S\subset E^d$. To show that $g_i$ satisfies \textbf{Property A}, that is $g_i(H)\subset [-O(n^\frac{d-1}{d+1}),\sqrt{d}n]$, we need to show that $\Delta(H_1,\alpha_i)\leq O(n^\frac{d-1}{d+1})=o(n)$ and $\Delta(H_2,\alpha_i)\leq \sqrt{d}n$ as in the proof of \Cref{g_i(H)}. Hence all points in $H_1$ must be close to $\alpha_i$. This is intuitively true since before the perturbation $\alpha_0$ is tangent to the $d$-sphere $S$, and  $S$ is approximately tangent to $H$ (see \Cref{DPS}).


\vspace{0.5cm}
\begin{lemma}\label{g_i(H)} Given the dimension $d\geq 3$,
when $n$ is sufficiently large,
\[g_i(H)\subset [-a,b], 1\leq i\leq d-1.\]
\end{lemma}
\begin{proof}
Since $\va_i$ is a unit normal vector of $\alpha_i$ beginning from $\vh\in\alpha_i$, and $H=H_1\cup H_2$ with points of $H_2$ lying on the same side of $\va_i$  and $H_1$ the other side,
\[\min g_i(H)=\min \va_i\cdot(H_1-\vh)\geq-\|\va_i\|\cdot \Delta(H_1,\alpha_i)=-\Delta(H_1,\alpha_i),\]
\[\max g_i(H)=\max \va_i\cdot(H_2-\vh)\leq\|\va_i\|\cdot \Delta(H_2,\alpha_i)=\Delta(H_2,\alpha_i).\]
Then $g_i(H)\subset[-\Delta(H_1,\alpha_i),\Delta(H_2,\alpha_i)]$. Hence it suffices to show that $ \Delta(H_1,\alpha_i)\leq a=O(n^\frac{d-1}{d+1})$ and $\Delta(H_2,\alpha_i)\leq b=\sqrt{d}n$. The latter inequality is trivial since $H$ is of diameter less than $\sqrt{d}n$. It is left to prove $ \Delta(H_1,\alpha_i)\leq a$. To simplify the computation, we first modify the basis $\vy_1,\ldots,\vy_{d-1}$ by choosing a proper $\vy_i$ from each pair $\{\vy_i,-\vy_i\}$ as below.

{

Let $\pi_i:E^d\rightarrow U_i^\bot$ be the projection on $U_i^\bot$. Then $\pi_i(U_i)=\vh$ and $\pi_i(\alpha_0)=l_i$ is a line as shown in \Cref{Fig.4}. Recall that $\alpha_0$ is tangent to $S$. Let $\mathbf{t}=\alpha_0\cap S$ be the tangent point, then $\pi_i(\mathbf{t})\in l_i$. Since $\va_i$ is the normal vector of $\alpha_i$ and $\va_i\bot U_i$, we have $U_i\subset\alpha_i$. Since $\pi_i(U_i)$ degenerates into a point, $\pi_i(\alpha_i)$ degenerates into a line in $U_i^\bot$. Let $l_{\alpha_i}=\pi_i(\alpha_i)$ denote the line and let $\overline{S}$ denote the boundary circle of $\pi_i(S)$. Now take $\vy_i$ from $\{\vy_i,-\vy_i\}$ such that $\vh_i$ points to the same direction as $\pi_i(\mathbf{t})-\vh$. Then as shown in \Cref{Fig.4}, $l_{\alpha_i}$ slices the circle $\overline{S}$ into two arcs with the minor arc of central angle not more than $2\theta_i$ (the central angle is exactly $2\theta_i$ if and only if $\pi_i(\mathbf{t})=\vh$), which is needed in \textbf{Case 2} later.}

Let $l_{\va_i}$ denote the line spanned by $\va_i$, which is a $1$-subspace of $E^d$. Then $l_{\va_i}\subset U_i^\bot$ since $\va_i\in U_i^\bot$. Let $\tau_i:E^d\rightarrow l_{\va_i}$ be the projection on $l_{\va_i}$, then for any point $\vx\in E^d$,
\[d(\vx,\alpha_i)=d(\tau_i(\vx),\tau_i(\alpha_i))=d(\tau_i(\vx),\vh).\]
This implies $\Delta(H_1, \alpha_i)=\max_{\vx\in H_1}d(\tau_i(\vx),\vh)$.

Since $l_{\alpha_i}\bot l_{\va_i}$, we have $\tau_i(l_{\alpha_i})=\vh$. Since $l_{\va_i}\in U_i^\bot$, we have $\tau_i=\tau_i\circ\pi_i$. Then
\[d(\tau_i(\vx),\vh)=d(\tau_i\circ\pi_i(\vx),\tau_i(l_{\alpha_i}))=d(\pi_i(\vx),l_{\alpha_i}),\]
and $\Delta(H_1, \alpha_i)=\max_{\vx\in H_1}d(\pi_i(\vx),l_{\alpha_i})={\Delta}(\pi_i(H_1),l_{\alpha_i})$. 
 Therefore, it suffices to compute ${\Delta}(\pi_i(H_1),l_{\alpha_i})$ in $U_i^\bot$.

Recall that the circle $\overline{S}$ is the boundary of $\pi_i(S)$ and let the polygon $\overline{P}$ be the boundary of $\pi_i(P)$. Assume that the line $l_{\alpha_i}$ divides $\overline{P}$ into $\overline{P_1}$ and $\overline{P_2}$, divides $\overline{S}$ into $\overline{S_1}$ and $\overline{S_2}$, where $\overline{P_1}$ and $\overline{S_1}$ are on the same side as $\pi_i(H_1)$ of $l_{\alpha_i}$, as shown in \Cref{Fig.4}.

\textbf{Case 1.} $\overline{S_1}=\emptyset$. Then $\pi_i(S)$ and $\pi_i(H_1)$ are on different sides of $l_{\alpha_i}$, then $S$ and $H_1$ are on different sides of $\alpha_i$, then
\[\Delta(H_1,\alpha_i)\leq \Delta(H_1,S)\leq \Delta(P,S),\]
where the second inequality holds since $H_1$ is inside $P$ and outside $S$.

\textbf{Case 2.} $\overline{S_1}\neq\emptyset$. Then $\overline{S_1}$ and $\overline{S_2}$ are two nontrivial arcs. {By the choice of $\vy_i$ above}, $\overline{S_1}$ has a central angle not more than $2\theta_i$, then $\Delta(\overline{S_1},l_{\alpha_i})\leq\frac{\sqrt{d}}{2}n(1-\cos\theta_i)$. Then
\[\Delta(H_1,\alpha_i)={\Delta}(\pi_i(H_1),l_{\alpha_i})\leq {\Delta}(\overline{P_1},l_{\alpha_i})\leq \Delta(\overline{P},\overline{S})+\Delta(\overline{S_1},l_{\alpha_i})\leq \Delta(P,S)+\frac{\sqrt{d}}{2}n(1-\cos\theta_i).\]

Combining these two cases, we have
\[\Delta(H_1,\alpha_i)\leq \Delta(P,S)+\frac{\sqrt{d}}{2}n(1-\cos\theta_i).\]
By \Cref{DPS}, $\Delta(P,S)\leq (1+o(1))\frac{1}{4\sqrt{d}}n^\frac{d-1}{d+1}$. On the other hand, by \Cref{h_i}, $\|\vh_i\|\geq \vartheta_d^{-1}\lambda=\Omega(n^{\frac{1}{d+1}})$, then $\sin\theta_i=\|\vh_i\|^{-1}=o(1)$, $\cos\theta_i=1-o(1)$. Hence
\[{\frac{\sqrt{d}}{2}n(1-\cos\theta_i)=\frac{\sqrt{d}}{2}n\frac{\sin^2\theta_i}{1+\cos\theta_i}=\frac{\sqrt{d}}{2}n\frac{\|\vh_i\|^{-2}}{2-o(1)}\leq(1+o(1))\frac{\sqrt{d}}{4}\vartheta_d^2n^\frac{d-1}{d+1}.}\]
Observe that $\frac{1}{4\sqrt{d}}\leq \frac{\pi^2}{192}\sqrt{d}\vartheta_d^2$ when $d\geq 3$, finally we have
\begin{align*}
\Delta(H_1,\alpha_i)&\leq \Delta(P,S)+\frac{\sqrt{d}}{2}n(1-\cos\theta_i)
\leq(1+o(1))\frac{1}{4\sqrt{d}}n^\frac{d-1}{d+1}+(1+o(1))\frac{\sqrt{d}}{4}\vartheta_d^2n^\frac{d-1}{d+1}\\
&<\frac{10}{33}\sqrt{d}\vartheta_d^2n^\frac{d-1}{d+1}=a.
\end{align*}
\end{proof}
\vspace{0.5cm}

Second, we show that \textbf{Property B} holds, i.e., the map $(g_0,g_1,\ldots,g_{d-1})$ is injective on $E^d$.

Recall that $\vh$ is the origin and $\vy_1,\ldots,\vy_{d-1}$ is a basis of $L(\va_0)$, then $L(\va_0)=\vh+\bbZ\vy_1+\cdots+\bbZ\vy_{d-1}$ and each $\vx\in L(\va_0)$ has a unique integral coordinate representation $\vx=(x_1,\ldots,x_{d-1})$ such that $\vx=\vh+x_1\vy_1+\cdots+x_{d-1}\vy_{d-1}$. The following lemma reveals the relationship between the direction function $g_i$ and the coordinate $(x_1,\ldots,x_{d-1})$.

\vspace{0.5cm}
\begin{lemma}\label{g_i}
For each $\vx\in L(\va_0)$ of coordinate $(x_1,\ldots,x_{d-1})$ with respect to the basis $\vy_1,\ldots,\vy_{d-1}$, we have $g_i(\vx)=x_i$ for $1\leq i\leq d-1$.
\end{lemma}
\begin{proof}
Since  $\va_i\in U_i^\bot$ and $\vy_j\in U_i$ if $ j\neq i$, we have $\va_i\bot\vy_j$ for $j\neq i$. Recall $\vh_i\in U_i^\bot$. Take the orthogonal decomposition $\vy_i=\vh_i+\vy'_i$, where $\vy'_i\in U_i$. From $\vx-\vh=x_1\vy_1+\ldots+x_{d-1}\vy_{d-1}$ and $\va_i\bot\vy_j~(\forall j\neq i)$, we get
\[g_i(\vx)=\va_i\cdot(\vx-\vh)=x_i\va_i\cdot\vy_i
=x_i\va_i\cdot\vh_i=x_i\|\va_i\|\|\vh_i\|\cos\angle(\va_i,\vh_i)
=x_i\|\vh_i\|\sin\theta_i=x_i\|\vh_i\|\|\vh_i\|^{-1}=x_i.\]
\end{proof}

\begin{remark}\label{rem}
Similarly, since $\vy_1,\ldots,\vy_{d-1}$ is a basis of $\alpha_0$, for each $\vx\in\alpha_0$ of coordinate $(x_1,\ldots,x_{d-1})$ with respect to the basis, $g_i(\vx)=x_i$ holds for $1\leq i\leq d-1$. Then
\[\vx\neq\vy\in\alpha_0~\Leftrightarrow~(x_1,\ldots,x_{d-1})\neq(y_1,\ldots,y_{d-1})~\Leftrightarrow~
(g_1(\vx),\ldots,g_{d-1}(\vx))\neq(g_1(\vy),\ldots,g_{d-1}(\vy)).\]
This implies that the map $(g_1,\ldots,g_{d-1}):E^d\rightarrow E^{d-1},\vx\mapsto (g_1(\vx),\ldots,g_{d-1}(\vx))$ is injective restricted on $\alpha_0$. By the linearity of $g_i$ and that $\alpha_0$ is a level set of $g_0$, we immediately arrive at \textbf{Property B} that the map $(g_0,g_1,\ldots,g_{d-1}):E^d\rightarrow E^d,~\vx\mapsto (g_0(\vx),g_1(\vx),\ldots,g_{d-1}(\vx))$ is injective.

Intuitively, this conclusion indeed results from that  direction vectors $\va_i$ are constructed based on the basis $\vy_1,\ldots,\vy_{d-1}$ so that $\va_i$ inherit the linear independence of this basis.
\end{remark}
\begin{remark} Since $L(\va_0)$ is the maximal integral $(d-1)$-lattice in $\alpha_0= g_0^{-1}(0)$,
by \Cref{g_i}, we have $g_i(H\cap g_0^{-1}(0))=g_i(H\cap L(\va_0))\subset\bbZ,1\leq i\leq d-1$. That is, \textbf{Property C} holds for the case $k=0$.
\end{remark}

\vspace{0.5cm}
Third, we show that $g_i$ satisfies \textbf{Property C}, i.e., $g_i(H\cap g_0^{-1}(k))$ is contained in a translate of $\bbZ$ for $1\leq i\leq d-1,k\in \bbZ$.

When $k=0$, we already have $g_i(H\cap g_0^{-1}(0))\subset\bbZ$. For general $k\in \bbZ$, denote
\[P_k:=g_0^{-1}(k),\text{ and }L_k:=P_k\cap\bbZ^d.\]
Then $P_k\in \P(\va_0)$ is an integer-valued level set of $g_0$ and thus $P_k$ is parallel to $\alpha_0$. In particular, $P_0=\alpha_0$ and $L_0=L(\va_0)$. Since $H\subset \bbZ^d$, \textbf{Property C} can be written as $g_i(H\cap L_k)\subset\bbZ+\epsilon_{ik}$.

We claim that $L_k$ is a translate of the lattice $L(\va_0)$. By the periodicity of the lattice $\bbZ^d$, and $L(\va_0)\subset\bbZ^d$ and that $P_k$ is parallel to $\alpha_0$, it suffices to show that there exists an integer point $\vx_k$ in $P_k$. Indeed, since $\va_0$ is primitive, there exists an integer point $\vy$ such that $\va_0\cdot\vy=1$. Take the integer point $\vx_k=(k+\va_0\cdot\vh)\vy$ to get $g_0(\vx_k)=\va_0\cdot((k+\va_0\cdot\vh)\vy-\vh)=k$, which implies that $\vx_k\in P_k$.

By the claim, $\vy_1,\ldots,\vy_{d-1}$ is also a basis of each lattice $L_k$.

By $g_0(H)\subset\bbZ$ from \Cref{g}, we have $H\subset\bigcup_{k\in\bbZ}L_k$. In fact, $k$ belongs to $ [0,\sqrt{d}nR) \cap \bbZ$. Hence we can slice $H$ into lattices $L_k$ and study each $g_i(L_k)$ to get the range of $g_i(H)$.
The following lemma tells that $g_i(L_k)$ has a good distribution and then \textbf{Property C} holds.
~\\
\begin{lemma}\label{g_i L_k}
(1) For each $k\in \bbZ$, there exists a point $\mathbf{o}_k\in L_k$ such that $g_i(\mathbf{o}_k)=\epsilon_{ik}\in(-\frac{1}{2},\frac{1}{2}]$ for $1\leq i\leq d-1$.

(2) Suppose the basis $\vy_1,\ldots,\vy_{d-1}$ of $L_k$ begins from $\mathbf{o}_k$, then for each $\vx\in L_k$ of coordinate $(x_1,\ldots,x_{d-1})$ with respect to the basis, we have $g_i(\vx)=x_i+\epsilon_{ik}$ for $1\leq i\leq d-1$.
\end{lemma}
\begin{proof}
(1) Take an arbitrary point $\vx_0\in L_k$ and consider $g_i(\vx_0)$. Write $g_i(\vx_0)=m_i+\epsilon_{ik}$ such that $m_i\in\bbZ$ and $\epsilon_{ik}\in(-\frac{1}{2},\frac{1}{2}]$ for $1\leq i\leq d-1$. Let $\mathbf{o}_k=\vx_0-m_1\vy_1-\cdots-m_{d-1}\vy_{d-1}\in L_k$. Then for $1\leq i\leq d-1$, since $\va_i\cdot\vy_j=\delta_{ij}$ (by the proof of \Cref{g_i}),
\[g_i(\mathbf{o}_k)=\va_i\cdot(\vx_0-\sum_{j=1}^{d-1}m_j\vy_j-\vh)
=g_i(\vx_0)-\sum_{j=1}^{d-1}m_j\va_i\cdot\vy_j=g_i(\vx_0)-m_i=\epsilon_{ik}.\]

(2) Similar to the proof of \Cref{g_i}.
\end{proof}

\vspace{0.3cm}

\Cref{g_i L_k} implies that $g_i(L_k)\subset\bbZ+\epsilon_{ik}$ and thus $g_i(H\cap L_k)\subset g_i(L_k)\subset\bbZ+\epsilon_{ik}$, which is \textbf{Property C}. Compose $g_i$ into $f_i$ to get the $i$-th polynomial $p_i=f_i\circ g_i$, $1\leq i\leq d-1$, and finally let $p=p_0p_1\cdots p_{d-1}$.

\vspace{0.5cm}
\section{Proof of \Cref{d_polynomial}}
In this section, we prove that the polynomial $p=p_0p_1\cdots p_{d-1}$  constructed in Section~\ref{sec:con} satisfies \Cref{d_polynomial}. That is, the polynomial $p$ has the following properties:
\begin{itemize}
  \item[(a)]$p$ has a peak value at the point $\vh$;
  \item[(b)]$p$ has a low degree $O(n^\frac{d}{d+1})$.
\end{itemize}

\subsection{Peak value}
The idea to show that $p$ has a peak value at the point $\vh$ is as follows. First, we show that $|p(\vx)|$ must have a small summation over points in each slice $L_k\cap H$ with $k\geq 1$; see \Cref{p slice}. Note that when $k\geq 1$, each summation excludes the value at the point $\vh$, which is in the slice $L_0\cap H$. Next, we summarize $|p(\vx)|$ over all slices of $H$ to see that $p$ has a peak value $p(\vh)$; see \Cref{p(H)}.

\vspace{0.5cm}
\begin{lemma}\label{p slice}
For each integer $k\in[1,\sqrt{d}nR)$, we have
\[\sum_{\vx\in L_k\cap H}|p(\vx)|<\frac{1}{k^2}(4m+\pi^2-4)^{d-1}.\]
\end{lemma}
\begin{proof}
By \Cref{g_i L_k}, in the lattice $L_k$, there exists a $(d-1)$-tuple $(\epsilon_{1k},\ldots,\epsilon_{d-1,k})\in(-\frac{1}{2},\frac{1}{2}]^{d-1}$ such that for each $\vx\in L_k$ of coordinate $(x_1,\ldots,x_{d-1})\in\bbZ^{d-1}$ with respect to the basis $\vy_1,\ldots,\vy_{d-1}$ and the corresponding origin $\mathbf{o}_k$, we have $g_i(\vx)=x_i+\epsilon_{ik},1\leq i\leq d-1$. Then
\[p_0(\vx)=f_0\circ g_0(\vx)=f_0(k)\leq\frac{1}{k^2},\quad p_i(\vx)=f_i\circ g_i(\vx)=f_i(x_i+\epsilon_{ik}),1\leq i\leq d-1.\]
Recall $f_1=f_2=\cdots=f_{d-1}=f$. Then
\begin{align*}
\sum_{\vx\in L_k\cap H}|p(\vx)|&\leq \sum_{\vx\in L_k}|p(\vx)|=\sum_{\vx\in L_k}|p_0(\vx)||p_1(\vx)|\cdots|p_{d-1}(\vx)|=|f_0(k)|\sum_{\vx\in L_k}\prod_{i=1}^{d-1}|f(x_i+\epsilon_{ik})|\\
&\leq\frac{1}{k^2}\prod_{i=1}^{d-1}\sum_{j\in\bbZ}|f(j+\epsilon_{ik})|
<\frac{1}{k^2}\prod_{i=1}^{d-1}\left(4m+\sum_{j=1}^{+\infty}\left(\frac{1}{(j+\epsilon_{ik})^2}+\frac{1}{(-j+\epsilon_{ik})^2}\right)\right),
\end{align*}
where the second inequality results from \textbf{Property B} and the last inequality applies \Cref{f2} (b).
Consider a function $h(x)=\frac{1}{(j+x)^2}+\frac{1}{(-j+x)^2}=\frac{2(j^2+x^2)}{(j^2-x^2)^2}$ on $(-\frac{1}{2},\frac{1}{2}]$, where $j$ is a positive integer. Obviously $h(x)\leq h(\frac{1}{2})=4(\frac{1}{(2j+1)^2}+\frac{1}{(2j-1)^2})$. Hence
\begin{align*}
\sum_{\vx\in L_k\cap H}|p(\vx)|&<\frac{1}{k^2}\prod_{i=1}^{d-1}\left(4m+4(2\sum_{j=1}^{+\infty}\frac{1}{(2j-1)^2}-1)\right)=\frac{1}{k^2}\prod_{i=1}^{d-1}\left(4m+4(2\cdot\frac{\pi^2}{8}-1)\right)\\
&=\frac{1}{k^2}(4m+\pi^2-4)^{d-1}.
\end{align*}

\end{proof}

\vspace{0.5cm}
\begin{lemma}\label{p(H)}Given the dimension $d\geq 3$,
when $n$ is sufficiently large, the polynomial $p$ satisfies 
\[p(\vh)>\sum_{\vx\in H\setminus\vh}|p(\vx)|.\]
\end{lemma}
\begin{proof}
Since $g_i(\vh)=0,0\leq i\leq d-1$, we have $p(\vh)=f_0(0)f_1(0)\cdots f_{d-1}(0)=m_0m^{d-1}>0$.
For the slice $L_0\cap H$, similar to the proof of \Cref{p slice}, apply \Cref{g_i} to get
\[\sum_{\vx\in L_0\cap H}|p(\vx)|<m_0(m+\frac{\pi^2}{3})^{d-1}.\]
Recall $m_0=\frac{\pi^2}{6}4^d,m=\frac{2}{3}\pi^2(d-1)$, then
\begin{align*}
p(\vh)-\sum_{\vx\in H\setminus\vh}|p(\vx)|&=2p(\vh)-\sum_{\vx\in H}|p(\vx)|=2m_0m^{d-1}-\sum_{\vx\in L_0\cap H}|p(\vx)|-\sum_{k=1}^{\lfloor\sqrt{d}nR\rfloor}\sum_{\vx\in L_k\cap H}|p(\vx)|\\
&\overset{\Cref{p slice}}{>}2m_0m^{d-1}-m_0(m+\frac{\pi^2}{3})^{d-1}-(4m+\pi^2-4)^{d-1}\sum_{k=1}^{\lfloor\sqrt{d}nR\rfloor}\frac{1}{k^2}\\
&>2m_0m^{d-1}-m_0(m+\frac{\pi^2}{3})^{d-1}-\frac{\pi^2}{6}(4m+\pi^2-4)^{d-1}\\
&=m_0m^{d-1}\left(2-(1+\frac{1}{2(d-1)})^{d-1}-\frac{1}{4}(1+\frac{\pi^2-4}{\frac{8\pi^2}{3}(d-1)})^{d-1}\right)\\
&>m_0m^{d-1}\left(2-e^\frac{1}{2}-\frac{1}{4}e^\frac{\pi^2-4}{8\pi^2/3}\right)>0.
\end{align*}
\end{proof}

\subsection{Low degree}
The degree of $p$ is just estimated by applying \Cref{f1} and \Cref{f2} with the predefined parameters.

\begin{lemma}\label{deg p}
When $n$ is sufficiently large and the dimension $d\geq 3$,  the polynomial $p$ has degree
\[\deg p< d^{3d/2}n^\frac{d}{d+1}-d.\]
\end{lemma}
\begin{proof}
Since $g_i$ are linear functions, we have $\deg p=\sum_{i=0}^{d-1}\deg f_i$.
By \Cref{f1}, for $b_0=\sqrt{d}nR$ and $m_0=\frac{\pi^2}{6}4^d$,
\[\deg f_0<\sqrt{\pi}\sqrt{b_0}\sqrt[4]{m_0}+2=\sqrt{\pi}\sqrt{\sqrt{d}n\cdot \sqrt{d}n^\frac{d-1}{d+1}}\sqrt[4]{\frac{\pi^2}{6}4^d}+2=(1+o(1))\frac{\pi}{6^{1/4}}\sqrt{d}2^{d/2}n^\frac{d}{d+1}.\]
By \Cref{f2}, for $a=\frac{10}{33}\sqrt{d}\vartheta_d^2n^\frac{d-1}{d+1}$, $b=\sqrt{d}n$ and $m=\frac{2\pi^2}{3}(d-1)$, when $1\leq i\leq d-1$,
\begin{align*}
\deg f_i<7\sqrt{abm}+2=7\sqrt{\frac{10}{33}\sqrt{d}(\frac{8\sqrt{2}}{81\pi})^2\frac{d^{3d}}{d(d-1)^3}n^\frac{d-1}{d+1}\cdot\sqrt{d}n\cdot\frac{2\pi^2}{3}(d-1)}+2<0.44\frac{d^{3d/2}}{d-1}n^\frac{d}{d+1}.
\end{align*}
Observing that $\frac{\pi}{6^{1/4}}\sqrt{d}2^{d/2}< 0.1\cdot d^{3d/2}$ for $3\leq d\in \bbZ$, we have
\[\deg p=\sum_{i=0}^{d-1}\deg f_i<(0.1+0.44)d^{3d/2}n^\frac{d}{d+1}<d^{3d/2}n^\frac{d}{d+1}-d.\]
\end{proof}
\vspace{0.3cm}

Combining \Cref{p(H)} and \Cref{deg p}, we complete the proof of \Cref{d_polynomial}.

\vspace{0.5cm}

\section{An improvement for matrix reconstruction}
In this section, we revisit the matrix reconstruction problem from submatrices. We prove \Cref{2_polynomial}, which improves the result in \cite{kos2009reconstruction} by reducing the constant factor  from $38$  to $6.308$ for both  symmetric and nonsymmetric cases. This improvement is obtained by using a different way of constructing two direction functions.


Similar to the case $d\geq 3$, we need the following auxiliary univariate polynomial from \cite{foster2000improvement}, which takes a peak value at the origin.

\vspace{0.5cm}
\begin{lemma}[\cite{foster2000improvement}]\label{f}
For any positive integer $N$, there exists a polynomial $f(x)$ with real coefficients such that 
\begin{itemize}
  \item[(a)] $f(0)>\sum_{i=1}^{N+1}|f(i)|$, and
  \item[(b)] $\deg f\leq 2\lfloor\sqrt{N\ln2}\rfloor+2$.
\end{itemize}
\end{lemma}
\vspace{0.5cm}

\begin{definition}
For two integer points $\va,\vb\in \bbZ^2$, $\va\vb$ is a \emph{primitive edge} if the segment $\va\vb$ has no interior integer points, i.e., $\vb$ is a primitive point when fixing $\va$ as the origin.
\end{definition}
\vspace{0.3cm}

We still denote by $\N(R)$ the set of primitive points of length not more than $R$, and let $N(R)=|\N(R)|$. Each primitive point $\va\in \N(R)$ corresponds to a primitive edge $\vo\va$ with the origin $\vo$ as one endpoint. In this sense, $\N(R)$ is also the set of primitive edges. Denote by $l(R)$ the total length of all primitive edges in $\N(R)$. We estimate the value of $l(R)$ as below.
\vspace{0.5cm}
\begin{lemma}\label{l(R)}
For large $R$, $l(R)=(\frac{4}{\pi}+o(1))R^3$.
\end{lemma}
\begin{proof}
Since $l(R)$ is an increasing function, it suffices to verify the case that $R$ is a positive integer.
It is known that $N(x)=\frac{6}{\pi}x^2+o(x)$ \cite[Section 1]{wu2002primitive}, then $N(x)-N(x-1)=\frac{12}{\pi}x+o(x)$.
Hence
\begin{align*}
l(R)&=\sum_{x=1}^{R}(N(x)-N(x-1))(x+O(1))=\sum_{x=1}^{R}(\frac{12}{\pi}x+o(x))(x+O(1))=\frac{12}{\pi}\sum_{x=1}^{R}x^2+\sum_{x=1}^{R}o(x^2)\\
&=\frac{12}{\pi}\frac{R(R+1)(2R+1)}{6R}+o(R^3)=(\frac{4}{\pi}+o(1))R^3.
\end{align*}
\end{proof}
\vspace{0.5cm}

Given an arbitrary nonempty set $H\subset [n]^2$, let $P$ be the smallest convex polygon enclosing $H$ so that its vertices $V(P)\subset H$. In the following we estimate the value of $|V(P)|$, i.e., the size of $V(P)$.

\vspace{0.5cm}
\begin{lemma}\label{V(P)}
When $n$ is sufficiently large, $|V(P)|\leq (1+o(1))\frac{6}{\sqrt[3]{\pi}}n^\frac{2}{3}$.
\end{lemma}
\begin{proof}
Let $R=t(\pi n)^\frac{1}{3}$ with a constant $t>1$, then by \Cref{l(R)}, $l(R)=(1+o(1))4t^3n>4n$. On the other hand, the perimeter of $P$ is less than $4n$ since $P$ lies in a square of edge length $n-1$ and $P$ is a convex polygon. So $l(R)$ is larger than the perimeter of $P$.

Let $E(P)$ be the set of edges of $P$, then $|E(P)|=|V(P)|$. Since $V(P)\subset H\subset\bbZ^2$, each edge in $E(P)$ is an integral multiple of a primitive edge. Then mapping each edge in $E(P)$ to the corresponding primitive edge beginning from the origin, we get a set $E'(P)$ of primitive edges. Note that this map is injective, that is, $|E'(P)|=|E(P)|$. We claim that $|V(P)|<N(R)$. Suppose not, that is, $N(R)\leq |V(P)|=|E(P)|=|E'(P)|$. If $\N(R)\subset E'(P)$, then $l(R)$ is not larger than the perimeter of $P$, a contradiction. If  $\N(R)\not\subset E'(P)$, then we can one to one map each primitive edge in $\N(R)\setminus  E'(P)$ to a longer primitive edge in $E'(P)\setminus \N(R)$ randomly. So we must have $l(R)$ not larger than the perimeter of $P$ again, a contradiction. Therefore,
\[|V(P)|< N(R)=(1+o(1))\frac{6}{\pi}R^2=(1+o(1))\frac{6}{\sqrt[3]{\pi}}t^2n^\frac{2}{3}.\]
Take $t$ to be $ 1+o(1)$ to get $|V(P)|\leq (1+o(1))\frac{6}{\sqrt[3]{\pi}}n^\frac{2}{3}$.
\end{proof}
\vspace{0.5cm}

Similar to the case of general dimension, we construct the polynomial in \Cref{2_polynomial} with the form $p=p_1p_2$ and $p_i=f\circ g_i,~i=1,2$, where $f$ is the univariate polynomial provided by \Cref{f} and $g_i$ is a direction function with an expression $g_i(\vx)=\vu_i\cdot(\vx-\vh)$ for some $\vh\in V(P)\subset H$ and some primitive point $\vu_i$. To show that $p$ has a peak value $p(\vh)$, we expect that $p_i$ takes a unique large value on $\vh$, i.e., by \Cref{f} (a), $g_i(\vh)=0$ and $g_i(H\backslash\vh)>0$. Geometrically speaking, for each primitive point $\vu_i$, we expect to find a line $l_i\bot\vu_i$ such that $l_i\cap V(P)=\vh$ and $P$ lies exactly on the same side of $l_i$ as $\vu_i$. {We can apply Pigeonhole Principle to find out such two lines $l_1$ and $l_2$ as below.}

\vspace{0.5cm}
\begin{lemma}\label{two directions}
When $n$ is sufficiently large, let $R=\sqrt{3}t(\pi n)^\frac{1}{3}$ with a real constant $t>1$. Then there exists a point $\vh\in V(P)\subset H$ and two primitive points $\vu_1,\vu_2\in\N(R)$ such that  the direction functions $g_i(\vx)=\vu_i\cdot(\vx-\vh)~(i=1,2)$ have the following property:
\[g_i(\vh)=0,~g_i(H\backslash\vh)\subset (0,\sqrt{2}nR)\cap\bbZ,~~i=1,2.\]
\end{lemma}
\begin{proof}
For each $\va\in \N(R)$, consider  a line  $l\in\P(\va)=\{l:l\bot\va\}$ that is far away from $P$ against the direction of $\va$. Then we translate $l$ along the direction $\va$ such that the line touch $P$ for the first time (still use $l$ to denote this line).  Then $l\cap V(P)\neq\emptyset$ and $P$ lies exactly on the same side of $l$ as $\va$. Take some point $\vh\in l\cap V(P)$, then the direction function $g(\vx)=\va\cdot(\vx-\vh)$ satisfies $g(\vh)=0$ and $g(H)\geq 0$.

For each $\va\in \N(R)$, we get a triple $(\va,l,\vh)$ in this way. Denote by $T(R)$ the set of all such triples, i.e., $T(R)=\{(\va,l,\vh):\va\in \N(R)\}$. By \Cref{V(P)},
\[|T(R)|=N(R)=(1+o(1))\frac{6}{\pi}R^2=(1+o(1))3t^2\cdot\frac{6}{\sqrt[3]{\pi}}n^\frac{2}{3}\geq (1-o(1))3t^2|V(P)|>3|V(P)|.\]
By Pigeonhole Principle, there exist four triples in $T(R)$ containing the same point $\vh\in V(P)\subset H$. Since $\vh$ connects two edges in $E(P)$, we can find $(\vu_1,l_1,\vh)$ and $(\vu_2,l_2,\vh)$ among the four triples such that each $l_i$ does not coincide with an edge of $P$. Then for each $l_i$, $l_i\cap V(P)=l_i\cap H=\vh$. Hence for the two direction functions $g_i(\vx)=\vu_i\cdot(\vx-\vh)~(i=1,2)$,
\[g_i(\vh)=0,~g_i(H\backslash\vh)>0,~~i=1,2.\]
Similar to the proof of \Cref{g}, we easily have $g_i(H\backslash\vh)\subset (0,\sqrt{2}nR)\cap\bbZ,\quad i=1,2$.
\end{proof}

\vspace{0.5cm}
\begin{proof}[Proof of \Cref{2_polynomial}]

Let $R=\sqrt{3}t(\pi n)^\frac{1}{3}$ with a real constant $t>1$, then by \Cref{two directions}, we can find some $\vh\in V(P)\subset H$, two primitive points $\vu_1,\vu_2\in\N(R)$ and their direction functions $g_i(\vx)=\vu_i\cdot(\vx-\vh)~(i=1,2)$ such that
\[g_i(\vh)=0,g_i(H\backslash\vh)\subset (0,\sqrt{2}nR)\cap\bbZ,\quad i=1,2.\]
Take $t=1+o(1)$, then $R=(1+o(1))\sqrt{3}(\pi n)^\frac{1}{3}$. 
Construct a polynomial $p=p_1p_2$ with $p_i=f\circ g_i$, where $f$ is the polynomial provided by \Cref{f} for $N+1=\lfloor\sqrt{2}nR\rfloor$. Then $g_i(H\backslash\vh)\subset[N+1]$ and thus $p$ is well defined.

If $\vu_1=-\vu_2$, then $g_1=-g_2$. This contradicts with that $g_1(H\backslash\vh)>0$ and $g_2(H\backslash\vh)>0$. Hence $\vu_1\neq-\vu_2$, i.e., $\vu_1$ and $\vu_2$ are linearly independent. Define a map
\[\rho:H\backslash\vh\rightarrow [N+1]^2,~\vx\mapsto(g_1(\vx),g_2(\vx)),\]
then $\rho$ is injective. Hence
\begin{align}\label{eq-peak2}
\sum_{\vx\in H\backslash\vh}|p(\vx)|&=\sum_{\vx\in H\backslash\vh}|f\circ g_1(\vx)||f\circ g_2(\vx)|=\sum_{(i,j)\in\rho(H\backslash\vh)}|f(i)||f(j)| \nonumber\\
&\leq \sum_{i=1}^{N+1}|f(i)|\sum_{j=1}^{N+1}|f(j)| \nonumber\\
&<f^2(0)=f\circ g_1(\vh)\cdot f\circ g_2(\vh)=p(\vh).
\end{align}
Therefore, $p(\vh)$ is a peak value of $p$. Now we only need to show that $\deg p\leq6.308n^\frac{2}{3}-2$.

Since $g_1$ and $g_2$ are linear functions,
\begin{align*}
\deg p&=2\deg f\overset{\Cref{f}}{\leq} 2(2\lfloor\sqrt{\lfloor\sqrt{2}nR\rfloor\cdot\ln2}\rfloor+2)=2(2\lfloor\sqrt{\lfloor\sqrt{2}n\cdot(1+o(1))\sqrt{3}(\pi n)^\frac{1}{3}\rfloor\cdot\ln2}\rfloor+2)\\
&\leq6.308n^\frac{2}{3}-2.
\end{align*}
This completes the proof.
\end{proof}

\vspace{0.3cm}
{Before closing this section, we explain  a little bit why we are able to make an improvement on the constant coefficient in the previous result $k\geq 38n^\frac{2}{3}$ \cite{kos2009reconstruction}. This is due to two reasons.
\begin{enumerate} 
  \item We apply to a slightly better function in \Cref{f}.
  \item In the construction of the polynomial $p=p_1p_2$ by K\'{o}s \cite{kos2009reconstruction}, 
       {$p_2$ takes a large value $p_2(\vh)=f_2(0)$} on possibly more than one points $\vh\in H$. However, applying Pigeonhole Principle, we succeed in finding two polynomials
       $p_1$ and $p_2$, both of which have  {large} values on a unique point  $\vh$ in $H$,
      i.e., equations $g_1(\vx)=0$ and $g_2(\vx)=0$ both have a unique solution $\vh$ in $H$; see the proof of \Cref{two directions}. This arises a better result through the magnification in Eq. (\ref{eq-peak2}).
\end{enumerate}
}

\vspace{0.5cm}
\section{Conclusion}
 In this paper, we proved that if $k\geq d^{\frac{3}{2}d}n^\frac{d}{d+1}$ for fixed $d\geq 3$ and large $n$, then every hypermatrix $A\in \{0,1\}^{n^{\times d}}$ can be reconstructed by its $k$-deck $\M_k(A)$ (or $\M_k^{\text{p}}(A)$). To prove this, we simplified the problem by replacing the $k$-deck by its sum $S_k(A)$ (or $S_k^{\text{p}}(A)$) and reducing the reconstruction problem to a polynomial construction problem. Due to the limitation of our method, we see $n^\frac{d}{d+1}\rightarrow n$ as $d\rightarrow\infty$, which implies that our result performs badly for large dimension $d$. We look forward to further improvements on this result, and raise the following problem.
\vspace{0.3cm}
\begin{problem}\label{pr1}
  Does there exist an absolute constant $\alpha\in (0,1)$ such that for any fixed dimension $d\geq3$, we have $\kappa_d(n),\kappa_d^\text{p}(n)\leq c_dn^\alpha$, where $c_d$ is a constant only depending on $d$? That is, can all $n^{\times d}$-hypermatrices  be reconstructed by their (principal) $k$-decks when $k=\Omega(n^\alpha)$?
\end{problem}

\vspace{0.3cm}

Problem~\ref{pr1} can be solved if we can answer the following problem on the existence of certain polynomials.

\vspace{0.3cm}

\begin{problem}
  Does there exist an absolute constant $\alpha\in (0,1)$ such that for any fixed dimension $d\geq3$ and any nonempty set $H\subset [n]^d$, there exists a $d$-variable polynomial $p(x_1,\ldots,x_d)$ with a low degree $\deg p\leq c_dn^\alpha-d$ and a peak value on some point $\vh\in H$, i.e., $p(\mathbf{h})>\sum_{\mathbf{x}\in H\setminus\mathbf{h}}|p(\mathbf{x})|$?
\end{problem}

\vspace{0.3cm}

 Though we concerned the hypermatrices with a shape of the hypercube $[n]^d$, our proof actually works for hypermatrices with a general shape $[n_1]\times[n_2]\times\cdots\times[n_d]$. Since $[n_1]\times[n_2]\times\cdots\times[n_d]$ can be covered by a $d$-sphere of radius $\frac{1}{2}\sqrt{n_1^2+n_2^2+\cdots+n_d^2}$, by the same polynomial construction process, we immediately get the following conclusion.
\vspace{0.3cm}
\begin{corollary}\label{cor7}
For any fixed dimension $d\geq3$, given $d$ positive integers $n_1\geq n_2\geq\ldots\geq n_d$, when $n_d$ is sufficiently large and $n_d=\Omega(n_1^\frac{d}{d+1})$, we have $\kappa_d(n)\leq c'_d\sqrt{n_1^2+n_2^2+\cdots+n_d^2}^\frac{d}{d+1}$ for some constant $c'_d$. That is, all $n_1\times n_2\times\cdots\times n_d$-hypermatrices can be reconstructed by their $k$-decks when $k=\Omega(\sqrt{n_1^2+n_2^2+\cdots+n_d^2}^\frac{d}{d+1})$. 
\end{corollary}
\vspace{0.3cm}

Note that Corollary~\ref{cor7} requires  $n_1\geq n_2\geq\ldots\geq n_d$, and $n_d=\Omega(n_1^\frac{d}{d+1})$, which makes $\kappa_d(n)$ nontrivial. Therefore, Corollary~\ref{cor7} only works for the hypermatrices that close to a hypercube shape.

\vspace{0.5cm}
\section*{Appendix}
\subsection{Proof of \Cref{sym}}\label{app-1}

For each surjection $\tau:[d]\twoheadrightarrow[r]~(1\leq r\leq d)$, similar to $I_\tau$, we define
\[U_\tau=\{\vu\in [k]^d:\vu\text{ holds the same order as }\tau([d])\}.\]
Then for two distinct $\tau_1:[d]\twoheadrightarrow[r_1]$ and $\tau_2:[d]\twoheadrightarrow[r_2]$ ($r_1=r_2$ is allowable), $U_{\tau_1}\cap U_{\tau_2}=\emptyset$ and $[k]^d=\cup_{r=1}^d\cup_{\tau:[d]\twoheadrightarrow[r]}U_\tau$, i.e., $[k]^d$ is a disjoint union of all $U_\tau$.

Denote by $i_\tau$ the smallest number in $\tau^{-1}(i)$ for $1\leq i\leq r$. For each $(u_1,\ldots,u_d)\in U_\tau$, define
\[\gamma_{u_1\cdots u_d}(x_1,\ldots,x_d)=\binom{x_{1_\tau}-1}{u_{1_\tau}-1}\binom{x_{2_\tau}-x_{1_\tau}-1}{u_{2_\tau}-u_{1_\tau}-1}
\binom{x_{3_\tau}-x_{2_\tau}-1}{u_{3_\tau}-u_{2_\tau}-1}\cdots\binom{n-x_{r_\tau}}{k-u_{r_\tau}}.\]

By using the above notations, we give a proof of \Cref{sym}, which is a combination of the following two observations.

\vspace{0.3cm}
\begin{lemma}\label{entry2}
Let $A,B\in \{0,1\}^{n^{\times d}}$ and $D=A-B=(d_{i_1\cdots i_d})_{1\leq i_1,\ldots,i_d\leq n}$, then for each surjective map $\tau:[d]\twoheadrightarrow[r]~(1\leq r\leq d)$ and each $(u_1,\ldots,u_d)\in U_\tau$, the $(u_1,\ldots,u_d)$-entry of $S_k^{\text{p}}(D)$ is
\[\big(S_k^{\text{p}}(D)\big)_{u_1\cdots u_d}=\sum_{(i_1,\ldots,i_d)\in I_\tau}\gamma_{u_1\cdots u_d}(i_1,\ldots,i_d)\cdot d_{i_1\cdots i_d}.\]
\end{lemma}
\begin{proof}
After symmetrically deleting $n-k$ rows of each dimension of $D$, each remaining entry $d_{i_1\cdots i_d}$ preserves the same coordinate order as $(i_1,\ldots,i_d)$ in the principle $k$-deck $\M_k^{\text{p}}(D)$. Hence $d_{i_1\cdots i_d}$ never occurs in the entry $(u_1,\ldots,u_d)\in U_\tau$ if $(i_1,\ldots,i_d)\notin I_\tau$. It suffices to show that if $(i_1,\ldots,i_d)\in I_\tau$, each entry $d_{i_1\cdots i_d}$ occurs $\gamma_{u_1\cdots u_d}(i_1,\ldots,i_d)$ times in $\M_k^{\text{p}}(D)$ as the $(u_1,\ldots,u_d)$-entry.

Since $(i_1,\ldots,i_d)\in I_\tau$, $\{i_1,\ldots,i_d\}=\{i_{1_\tau},\ldots,i_{r_\tau}\}$, where $i_{1_\tau}<\cdots<i_{r_\tau}$. Then $d_{i_1\cdots i_d}$ becoming the $(u_1,\ldots,u_d)$-entry is equivalent to that $\{i_{1_\tau},\ldots,i_{r_\tau}\}$ becomes $\{u_{1_\tau},\ldots,u_{r_\tau}\}$ after deletion, where $u_{1_\tau}<\cdots<u_{r_\tau}$. It degenerates into a sequence deletion with $\binom{i_{1_\tau}-1}{u_{1_\tau}-1}\binom{i_{2_\tau}-i_{1_\tau}-1}{u_{2_\tau}-u_{1_\tau}-1}
\binom{i_{3_\tau}-i_{2_\tau}-1}{u_{3_\tau}-u_{2_\tau}-1}\cdots\binom{n-i_{r_\tau}}{k-u_{r_\tau}}=\gamma_{u_1\cdots u_d}(i_1,\ldots,i_d)$ choices. That is,
\[\big(S_k^{\text{p}}(D)\big)_{u_1\cdots u_d}=\sum_{(i_1,\ldots,i_d)\in I_\tau}\gamma_{u_1\cdots u_d}(i_1,\ldots,i_d)\cdot d_{i_1\cdots i_d}.\]
\end{proof}

\vspace{0.3cm}
\begin{lemma}\label{basis2}
For each surjective map $\tau:[d]\twoheadrightarrow[r]~(1\leq r\leq d)$ and all $(u_1,\ldots,u_d)\in U_\tau$, the polynomials $\gamma_{u_1\ldots u_d}(x_1,\ldots,x_d)$ form a basis of the linear space of polynomials in $r$ variables with total degree not more than $k-r$.
\end{lemma}
\begin{proof}
Obviously, the number of polynomials matches with the dimension of the linear space, which is $\binom{k}{r}$. So it suffices to prove linear independence.

Let $\lambda_{u_1\ldots u_d}$ for $ (u_1,\ldots,u_d)\in U_\tau$  be real numbers, not all zero; we have to show that
\[\sum_{(u_1,\ldots,u_d)\in U_\tau}\lambda_{u_1\ldots u_d}\gamma_{u_1\ldots u_d}(x_1,\ldots,x_d)\neq 0.\]

Let $(u_{10},\ldots,u_{d0})$ be the first set of indices in lexicographical order such that $\lambda_{u_{10}\ldots u_{d0}}\neq0$. This means that $\lambda_{u_1\ldots u_d}=0$ in every case when $u_1<u_{10}$, or $u_1=u_{10}$ and $u_2<u_{20}$, or so on. Substituting $(x_1,\ldots,x_d)=(u_{10},\ldots,u_{d0})$, by definition of $\gamma_{u_1\ldots u_d}$, we have $\gamma_{u_1\ldots u_d}(u_{10},\ldots,u_{d0})=0$ in every case when $u_1>u_{10}$, or $u_2-u_1>u_{20}-u_{10}$, or so on. The only case when $\lambda_{u_1\ldots u_d}\gamma_{u_1\ldots u_d}\neq 0$ is $(u_1,\ldots,u_d)=(u_{10},\ldots,u_{d0})$, Therefore,
\[\sum_{(u_1,\ldots,u_d)\in U_\tau}\lambda_{u_1\ldots u_d}\gamma_{u_1\ldots u_d}(u_{10},\ldots,u_{d0})=\lambda_{u_{10}\ldots u_{d0}}\gamma_{u_{10}\ldots u_{d0}}(u_{10},\ldots,u_{d0})=\lambda_{u_{10}\ldots u_{d0}}\binom{n-u_{d0}}{k-u_{d0}}\neq 0.\]
\end{proof}

\begin{proof}[Proof of \Cref{sym}]

Similar to the nonsymmetric case, \Cref{sym} follows from \Cref{entry2} and \Cref{basis2}.
\end{proof}
\vspace{0.3cm}


\end{document}